\documentclass[]{jgcc}
\pdfoutput=1
\usepackage{lastpage}
\jgccdoi{18}{2}{1}{17981}
\jgccheading{}{\pageref{LastPage}}{}{}{Apr.~8,~2026}{Apr.~21,~2026}{}

\usepackage{amssymb,amsmath}
\usepackage{hyperref}

\DeclareMathOperator{\Id}{Id}
\DeclareMathOperator{\charr}{char}
\DeclareMathOperator{\tr}{tr}
\DeclareMathOperator{\Ad}{Ad}
\DeclareMathOperator{\GL}{GL}
\DeclareMathOperator{\SL}{SL}
\DeclareMathOperator{\PGL}{PGL}

\DeclareMathOperator{\Rad}{Rad}
\DeclareMathOperator{\Aut}{Aut}
\DeclareMathOperator{\End}{End}
\DeclareMathOperator{\Inn}{Inn}

\numberwithin{equation}{section}

\newcommand{\diag}{\operatorname{diag}}
\newcommand{\Sha}{\operatorname{Sha}}

\begin{document}

\dedicatory{In memory of Evgeny Plotkin}

\title[Sha-rigidity of low-rank adjoint Chevalley groups]{Sha-rigidity of adjoint Chevalley groups of types $\mathbf A_1$, $\mathbf A_2$, $\mathbf B_2$, $\mathbf G_2$ over commutative rings}

\author[E.~Bunina]{Elena Bunina}
\address{Department of Mathematics, Bar--Ilan University, 5290002 Ramat Gan, Israel}

\author[V.~Kirakosyan]{Vazgen Kirakosyan}
\address{Faculty of Mechanics and Mathematics, Lomonosov Moscow State University, 119991 Moscow, Russia}

\author[R.~Treskunov]{Rachel Treskunov}
\address{Department of Mathematics, Bar--Ilan University, 5290002 Ramat Gan, Israel}


\begin{abstract}
We prove that every class-preserving endomorphism of the adjoint Chevalley group and of its elementary subgroup over a commutative ring is inner for the types A1, A2, and B2 when 2 is invertible, and for type G2 when 2 and 3 are invertible. Consequently, all these groups are Sha-rigid.
\end{abstract}

\maketitle

\section{Introduction}

Every inner automorphism is class-preserving.
In this paper we study the converse question: when is every class-preserving endomorphism inner?

Recall that an endomorphism $\varphi\in\End(G)$ of a group $G$ is called \emph{locally inner}
(or \emph{class-preserving}, cf.~\cite{Sah}) if for every $g\in G$ the elements $g$ and $\varphi(g)$
are conjugate in~$G$.
The goal of the present paper is to understand such endomorphisms for adjoint Chevalley groups
over commutative rings.

The problem is closely related to the local--global invariant introduced by Ono~\cite{BK112,BK115}.
Let $G$ act on itself by conjugation and consider the pointed set $H^{1}(G,G)$.
Among its elements one singles out those cohomology classes whose restrictions to every cyclic subgroup
of $G$ are trivial; this subset is denoted by $\Sha(G)$ and is called the \emph{Shafarevich--Tate set} of $G$.
We say that $G$ is \emph{$\Sha$-rigid} if $\Sha(G)$ is a singleton.

A convenient way to phrase $\Sha$-rigidity (going back to Mazur, see, e.g.,~\cite{BK115,Sah,BK})
is in terms of class-preserving morphisms.
Let $\End_{c}(G)$ (resp.\ $\Aut_{c}(G)$) denote the set of endomorphisms (resp.\ automorphisms)
$\varphi$ such that $\varphi(g)$ is conjugate to $g$ for all $g\in G$.
Then $\Sha(G)$ is trivial if and only if every class-preserving endomorphism is inner, i.e.
\[
\End_{c}(G)=\Inn(G),
\]
a property often referred to as \emph{Property~E}, see, e.g.,~\cite{BK}.
Clearly $\Inn(G)\subseteq \Aut_{c}(G)\subseteq \End_{c}(G)$.

Let $\Phi$ be a reduced irreducible root system and $R$ a commutative ring.
We write $G_{\mathrm{ad}}(\Phi,R)$ for the adjoint Chevalley group and
$E_{\mathrm{ad}}(\Phi,R)$ for its elementary subgroup.
For the types $\mathbf A_1$, $\mathbf A_2$, $\mathbf B_2$ we assume throughout that $2\in R^\times$,
while for the type $\mathbf G_2$ we assume in addition that $3\in R^\times$.
The main result of this paper concerns the four low-rank types
$\mathbf A_1$, $\mathbf A_2$, $\mathbf B_2$, $\mathbf G_2$,
for which we give separate direct proofs over arbitrary commutative rings.
In particular, the argument below is not obtained by reducing to the higher-rank local-ring case.

\begin{thm}\label{thm:main}
Let $R$ be a commutative ring.

\smallskip
\noindent\textup{(1)}
Assume that $2\in R^\times$ and let $\Phi\in\{\mathbf A_1,\mathbf A_2,\mathbf B_2\}$.
Then every locally inner endomorphism of $E_{\mathrm{ad}}(\Phi,R)$ is inner.
Moreover, any locally inner endomorphism of $G_{\mathrm{ad}}(\Phi,R)$ is inner for
$\Phi\in\{\mathbf A_2,\mathbf B_2\}$, and for $\Phi=\mathbf A_1$ under the additional assumption
that $E_{\mathrm{ad}}(\mathbf A_1,R)$ is normal in $G_{\mathrm{ad}}(\mathbf A_1,R)$.

\smallskip
\noindent\textup{(2)}
Assume that $2,3\in R^\times$.
Then every locally inner endomorphism of $E_{\mathrm{ad}}(\mathbf G_2,R)$ is inner, and every locally inner
endomorphism of $G_{\mathrm{ad}}(\mathbf G_2,R)$ is inner.

\smallskip
In particular, all groups listed above are $\Sha$-rigid.
\end{thm}

\medskip
\noindent\textbf{A remark on earlier work.}
For local rings and $\mathrm{rank}(\Phi)>1$, $\Sha$-rigidity of the elementary subgroup (under mild
small-denominator assumptions) was claimed in~\cite{BuninaKunyavskii2024}.
We would like to point out that the proof given there contains a gap:
since a locally inner endomorphism is automatically injective, the argument effectively requires
a structural description of \emph{injective endomorphisms} of Chevalley groups over local rings,
whereas the available input concerns \emph{automorphisms}.
At present, the needed description of injective endomorphisms in the required generality has not been established,
so the cited proof uses an ingredient that is not yet justified.
We expect that the higher-rank statement over local rings can nevertheless be recovered by the same general strategy,
provided one additionally verifies that each step remains valid for injective endomorphisms.
This issue, however, is separate from the proofs given in the present paper.

\medskip
\noindent\textbf{Independence of the present proof.}
The proofs below for the low-rank types $\mathbf A_1$, $\mathbf A_2$, $\mathbf B_2$, and $\mathbf G_2$
are independent of the argument in~\cite{BuninaKunyavskii2024}.
Even in the local-ring case, our proof is different and self-contained.
It does \emph{not} use a classification theorem for automorphisms,
nor any structural description of injective endomorphisms of Chevalley groups over local rings.
Thus the gap mentioned above does not affect the validity of the results proved here.

\medskip
\noindent\textbf{What is new here.}
In the present paper we prove $\Sha$-rigidity for the low-rank types
$\mathbf A_1$, $\mathbf A_2$, $\mathbf B_2$ over arbitrary commutative rings with $2\in R^\times$
and for $\mathbf G_2$ over arbitrary commutative rings with $2,3\in R^\times$
by a direct analysis of locally inner endomorphisms.
In particular, for these four root systems the argument is self-contained.
We also expect that the approach developed here extends to all reduced irreducible root systems over
commutative rings, although in several parts of the proof one will inevitably need
to treat different root systems separately.

\medskip\noindent\textbf{Organization.}
Section~2 fixes notation and recalls standard facts on Chevalley groups over commutative and local rings.
Sections~3--6 treat the types $\mathbf A_1$, $\mathbf A_2$, $\mathbf B_2$, and $\mathbf G_2$, respectively.

\bigskip

\section{Chevalley groups over commutative rings}\leavevmode\label{sec:Chev}

We fix a reduced irreducible root system $\Phi$ and a commutative ring $R$ with~$1$.
For any lattice $\mathcal P$ with $Q(\Phi)\subseteq \mathcal P \subseteq P(\Phi)$ we denote by
$G_{\mathcal P}(\Phi,R)$ the Chevalley group of type $(\Phi,\mathcal P)$ over $R$
(the group of $R$-points of the corresponding split Chevalley--Demazure group scheme).
The extreme cases $\mathcal P=Q(\Phi)$ and $\mathcal P=P(\Phi)$ give the adjoint and simply connected
forms, denoted $G_{\mathrm{ad}}(\Phi,R)$ and $G_{\mathrm{sc}}(\Phi,R)$, respectively.
Unless explicitly stated otherwise, we work with the adjoint form $G_{\mathrm{ad}}(\Phi,R)$
(and with its elementary subgroup), cf.~\cite{Carter,Steinberg}.

\medskip
\noindent\textbf{Root subgroups and the elementary subgroup.}
Fix a split maximal torus $T\le G=G_{\mathcal P}(\Phi,R)$ and identify $\Phi$ with $\Phi(G,T)$.
This determines the root subgroups $X_\alpha\le G$ ($\alpha\in\Phi$) and parametrizations
$x_\alpha\colon R\to X_\alpha$, $t\mapsto x_\alpha(t)$, satisfying the Chevalley commutator formula
(with integral structure constants), see~\cite{Carter,Steinberg}:
\begin{equation}\label{eq:Chev-comm}
[x_\alpha(t),x_\beta(u)]
=\prod_{i,j>0} x_{i\alpha+j\beta}\!\bigl(N_{\alpha\beta ij}\,t^i u^j\bigr),
\qquad (\alpha+\beta\neq 0).
\end{equation}
The \emph{elementary subgroup} is
\[
E(\Phi,R)=\bigl\langle\,x_\alpha(t)\ \bigm|\ \alpha\in\Phi,\ t\in R\,\bigr\rangle.
\]
When we need to emphasize the isogeny type, we write $E_{\mathrm{ad}}(\Phi,R)$ and $E_{\mathrm{sc}}(\Phi,R)$.

\medskip
\noindent\textbf{Standard subgroups.}
Choose a system of positive roots $\Phi^+$ with simple roots $\Delta$.
Put
\[
U=U(\Phi,R)=\langle x_\alpha(t)\mid \alpha\in\Phi^+,\ t\in R\rangle,\qquad
V=V(\Phi,R)=\langle x_{-\alpha}(t)\mid \alpha\in\Phi^+,\ t\in R\rangle.
\]
For $t\in R^\times$ and $\alpha\in\Phi$ set
\[
w_\alpha(t)=x_\alpha(t)x_{-\alpha}(-t^{-1})x_\alpha(t),\qquad
h_\alpha(t)=w_\alpha(t)w_\alpha(1)^{-1},
\]
and let $H=H(\Phi,R)=\langle\,h_\alpha(t)\mid \alpha\in\Phi,\ t\in R^\times\,\rangle$.
Then $H=T\cap E$, and $N_E(H)/H$ is canonically isomorphic to the Weyl group $W(\Phi)$,
see, e.g.,~\cite{Steinberg,Carter}.

\medskip
\noindent\textbf{Bruhat and Gauss decompositions.}
If $R$ is a field, every $g\in G(\Phi,R)$ admits the Bruhat decomposition
\begin{equation}\label{eq:Bruhat}
g=t\,u\,\mathbf w\,u', \qquad t\in T,\ \mathbf w\in W(\Phi),\ u,u'\in U,
\end{equation}
and $t\in H$ whenever $g\in E$, see, e.g.,~\cite{Carter,Steinberg}.
If $R$ is a (not necessarily field) \emph{local} ring, we will use the Gauss decomposition
\begin{equation}\label{eq:Gauss}
g=t\,u_1\,v\,u_2\qquad (t\in T,\ u_1,u_2\in U,\ v\in V),
\end{equation}
and for $g\in E$ one has $t\in H$, see, e.g.,~\cite{Abe1969}. We will only need existence of such factorizations
(and the standard uniqueness of the ordered factorization inside $U$ and inside $V$ once an order
on $\Phi^+$ is fixed).

\medskip
\noindent\textbf{Rank-one calculus.}
For any root $\alpha$ the subgroup $\langle X_\alpha,X_{-\alpha}\rangle$ is a rank-one
Chevalley subgroup. In particular, whenever $1+uv\in R^\times$ one has the standard identity
\begin{equation}\label{eq:rankone}
x_{-\alpha}(u)\,x_\alpha(v)
=
x_\alpha\!\left(\frac{v}{1+uv}\right)\,h_\alpha(1+uv)\,
x_{-\alpha}\!\left(\frac{u}{1+uv}\right),
\end{equation}
see, e.g.,~\cite{Steinberg}. Over a local ring $(R,\mathfrak m)$ this applies in particular to
$u\in\mathfrak m$, since $1+\mathfrak m\subseteq R^\times$.

\medskip
\noindent\textbf{Localization and reduction to local rings.}
For a prime ideal $\mathfrak p\subset R$ write $R_{\mathfrak p}=(R\setminus\mathfrak p)^{-1}R$.
If $\mathfrak m$ is maximal, then $R_{\mathfrak m}$ is a local ring with maximal ideal
$\Rad R_{\mathfrak m}$, and we denote the residue field by
$k(\mathfrak m)=R_{\mathfrak m}/\Rad R_{\mathfrak m}$.

\begin{prop}\label{inlocal}
Every commutative ring $R$ with $1$ embeds into the product of its localizations at maximal ideals
\[
S=\prod_{\mathfrak m\in\mathrm{MaxSpec}(R)} R_{\mathfrak m}
\]
via the diagonal map $a\mapsto \bigl(\frac{a}{1}\bigr)_{\mathfrak m}$,
see, e.g.,~\cite{AtiyahMacdonald}.
\end{prop}

By functoriality, this yields embeddings
$G(\Phi,R)\hookrightarrow \prod_{\mathfrak m} G(\Phi,R_{\mathfrak m})$
(and similarly for $E(\Phi,R)$). We will repeatedly reduce arguments to the local rings $R_{\mathfrak m}$
and further to the residue fields $k(\mathfrak m)$.

\medskip
\noindent\textbf{Normality of the elementary subgroup (rank $\ge 2$).}
If $\Phi$ has rank $\ge 2$, then $E(\Phi,R)\trianglelefteq G(\Phi,R)$ for all commutative rings $R$
(Suslin--Kopeiko--Taddei, see~\cite{v41}). For rank $1$ (type $\mathbf A_1$) normality may fail in general.

\medskip
\noindent\textbf{A trace identity (types $\mathbf A_2$ and $\mathbf B_2$).}
Let $\Ad\colon G_{\mathrm{ad}}(\Phi,R)\to \GL(\mathfrak g_R)$ be the adjoint representation.
For a long root $\alpha$ the trace $\tr(\Ad(x_\alpha(t)x_{-\alpha}(s)))$ is a polynomial in $s,t$ with integer
coefficients. In particular, for the two rank-$2$ types used in Sections~4 and~5 we have:
\begin{equation}\label{eq:traceA2B2}
\tr\!\bigl(x_\alpha(t)x_{-\alpha}(s)\bigr)=
\begin{cases}
s^2t^2-6st+8, & \Phi=\mathbf A_2,\\
s^2t^2-6st+10,& \Phi=\mathbf B_2,
\end{cases}
\qquad (s,t\in R),
\end{equation}
where the trace is taken in the adjoint representation.

\section{\texorpdfstring{$\Sha$-Rigidity of the adjoint Chevalley groups of type~$\mathbf A_1$}{Sha-Rigidity of the adjoint Chevalley groups of type A1}}\label{sec:A1}

Throughout this section $R$ is a commutative ring with $1/2\in R$, and
\[
E(\mathbf A_1)=E_{\mathrm{ad}}(\mathbf A_1,R)
\]
denotes the adjoint elementary Chevalley group of type $\mathbf A_1$ in the standard $3\times3$ realization.
It is generated by the root unipotents
\[
x_\alpha(t)=\begin{pmatrix}
1& t^2 & 2t\\
0& 1& 0\\
0& t& 1
\end{pmatrix},\qquad
x_{-\alpha}(t)=\begin{pmatrix}
1& 0& 0\\
t^2& 1& 2t\\
t& 0& 1
\end{pmatrix}\qquad (t\in R).
\]

\subsection*{Rank-one relations (used implicitly)}
We will use the standard relations in type $\mathbf A_1$:
\begin{enumerate}
\item $x_{\pm\alpha}(t)x_{\pm\alpha}(s)=x_{\pm\alpha}(t+s)$, hence $x_{\pm\alpha}(t)^{-1}=x_{\pm\alpha}(-t)$;
\item for $u\in R^*$,
\[
h_\alpha(u)x_\alpha(t)h_\alpha(u)^{-1}=x_\alpha(u^2t),\qquad
h_\alpha(u)x_{-\alpha}(t)h_\alpha(u)^{-1}=x_{-\alpha}(u^{-2}t);
\]
\item the standard rank-one factorization
\[
x_\alpha(t)\,x_{-\alpha}(s)
= x_{-\alpha}\!\left( \frac{s}{1+ts} \right)\,
h_\alpha(1+ts)\,
x_\alpha\!\left( \frac{t}{1+ts} \right)\qquad (1+ts\in R^*).
\]
\end{enumerate}

\subsection*{Set-up and normalization}
Let $\varphi\colon E(\mathbf A_1)\to E(\mathbf A_1)$ be a \emph{locally inner} endomorphism, i.e. $\varphi(g)$ is conjugate to $g$
for every $g\in E(\mathbf A_1)$.
Composing $\varphi$ with a suitable inner automorphism $i_y$ (conjugation by $y\in E(\mathbf A_1)$),
we may assume
\begin{equation}\label{eq:norm}
\varphi(x_\alpha(1))=x_{\alpha}(1).
\end{equation}
Then automatically $\varphi(x_\alpha(-1))=\varphi(x_\alpha(1)^{-1})=x_\alpha(-1)$.

\begin{rem}[Trace invariance]\label{rem:trace}
Since $\varphi(g)$ is conjugate to $g$ for every $g\in E(\mathbf A_1)$,
we will freely use $\tr(\varphi(X))=\tr(X)$ for all $X\in E(\mathbf A_1)$.
\end{rem}

\subsection*{\texorpdfstring{Step 1: Trace constraints for $\varphi(x_{-\alpha}(s))$}{Step 1: Trace constraints for phi(x-alpha(s))}}
Write
\[
\varphi(x_{-\alpha}(s))=
\begin{pmatrix}
a& b& c\\
d& e& f\\
g& h& i
\end{pmatrix},
\qquad a+e+i=\tr(\varphi(x_{-\alpha}(s)))=\tr(x_{-\alpha}(s))=3.
\]
A direct multiplication shows that for all $s,t\in R$,
\begin{equation}\label{eq:traceA1}
\tr\bigl(x_{-\alpha}(s)x_\alpha(t)\bigr)=s^2t^2+4st+3.
\end{equation}

Now evaluate at $t=\pm1$. Using \eqref{eq:norm} (and hence $\varphi(x_\alpha(\pm1))=x_\alpha(\pm1)$) we get
\[
\tr\bigl(\varphi(x_{-\alpha}(s)x_\alpha(\pm1))\bigr)
=\tr\bigl(\varphi(x_{-\alpha}(s))\,x_\alpha(\pm1)\bigr)
=(a+e+i)+d\pm(f+2g).
\]
Comparing with \eqref{eq:traceA1} for $t=\pm1$ yields the system
\[
d+(f+2g)=s^2+4s,\qquad d-(f+2g)=s^2-4s,
\]
hence
\begin{equation}\label{eq:d-g}
d=s^2,\qquad f+2g=4s.
\end{equation}

\subsection*{\texorpdfstring{Step 2: The centralizer of $x_\alpha(1)$}{Step 2: The centralizer of xalpha(1)}}
\begin{lem}\label{lem:cent}
The centralizer of $x_\alpha(1)$ in $\mathrm{M}_3(R)$ consists precisely of the matrices
\[
X=\begin{pmatrix}
a& b& 2h\\
0& a& 0\\
0& h& a
\end{pmatrix}\qquad (a,b,h\in R).
\]
\end{lem}

\begin{proof}
Let $X=(x_{ij})\in \mathrm{M}_3(R)$ commute with $x_\alpha(1)$.
Writing out $Xx_\alpha(1)=x_\alpha(1)X$ and comparing entries gives consecutively
$x_{21}=x_{23}=0$, then $x_{31}=0$, then $x_{22}=x_{11}=x_{33}$, and finally $x_{13}=2x_{32}$.
Renaming $a=x_{11}$, $b=x_{12}$, $h=x_{32}$ yields the claimed form.
\end{proof}

Since $x_\alpha(t)$ commutes with $x_\alpha(1)$, the element $\varphi(x_\alpha(t))$ commutes with
$\varphi(x_\alpha(1))=x_\alpha(1)$. Hence by Lemma~\ref{lem:cent},
\begin{equation}\label{eq:phi-xa}
\varphi(x_\alpha(t))=
\begin{pmatrix}
\alpha& \beta& 2\gamma\\
0& \alpha& 0\\
0& \gamma& \alpha
\end{pmatrix},
\qquad 3\alpha=\tr(\varphi(x_\alpha(t)))=\tr(x_\alpha(t))=3.
\end{equation}

\subsection*{\texorpdfstring{Step 3: Determining $\beta(t)$ and $\gamma(t)$}{Step 3: Determining beta(t) and gamma(t)}}
Apply trace invariance to $x_{-\alpha}(\pm1)x_\alpha(t)$:
\[
\tr\bigl(\varphi(x_{-\alpha}(\pm1)x_\alpha(t))\bigr)
=\tr\bigl(\varphi(x_{-\alpha}(\pm1))\,\varphi(x_\alpha(t))\bigr).
\]
By \eqref{eq:d-g} for $s=\pm1$, the matrix $\varphi(x_{-\alpha}(\pm1))$ satisfies $d=1$ and $f+2g=\pm4$,
so multiplying with \eqref{eq:phi-xa} and taking trace gives
\[
\tr\bigl(\varphi(x_{-\alpha}(\pm1))\,\varphi(x_\alpha(t))\bigr)=3\alpha+\beta\pm4\gamma=3+\beta\pm4\gamma.
\]
On the other hand, \eqref{eq:traceA1} gives $\tr(x_{-\alpha}(\pm1)x_\alpha(t))=t^2\pm4t+3$.
Therefore,
\begin{equation}\label{eq:beta-gamma}
\beta=t^2,\qquad \gamma=t.
\end{equation}

\subsection*{\texorpdfstring{Step 4: Normalizing $\varphi(x_{-\alpha}(1))$ by conjugation}{Step 4: Normalizing phi(x-alpha(1)) by conjugation}}
Write (using \eqref{eq:d-g} with $s=1$)
\[
\varphi(x_{-\alpha}(1))=
\begin{pmatrix}
a& b& c\\
1& e& 4-2g\\
g& h& i
\end{pmatrix},
\qquad a+e+i=3.
\]
Conjugating $\varphi$ by an element $x_\alpha(\mu)$ does not change $\varphi(x_\alpha(t))$
(because all matrices in Lemma~\ref{lem:cent} commute with $x_\alpha(\mu)$).
More precisely, replacing $\varphi$ by $i_{x_\alpha(\mu)^{-1}}\circ \varphi$ sends the $(3,1)$-entry $g$ to $g-\mu$.
Choosing $\mu=g-1$ we achieve $g=1$, and then \eqref{eq:d-g} forces $f=2$.
Thus we may assume
\begin{equation}\label{eq:phi-x-1}
\varphi(x_{-\alpha}(1))=
\begin{pmatrix}
a& b& c\\
1& e& 2\\
1& h& i
\end{pmatrix},
\qquad a+e+i=3.
\end{equation}

\subsection*{Step 5: A quadratic relation and its consequences}
Use the standard rank-one identity
\begin{equation}\label{eq:quad}
\bigl(x_\alpha (-1)\,x_{-\alpha}(1)\, x_\alpha (-1)\bigr)^2=1.
\end{equation}
Applying $\varphi$ to \eqref{eq:quad} and using \eqref{eq:norm}, \eqref{eq:phi-x-1}, we obtain
\[
1=\Bigl(x_\alpha(-1)\,\varphi(x_{-\alpha}(1))\,x_\alpha(-1)\Bigr)^2.
\]
Computing the conjugate $x_\alpha(-1)\,\varphi(x_{-\alpha}(1))\,x_\alpha(-1)$ explicitly and then squaring,
the $(2,1)$-entry gives $a+e=2$, hence $i=1$ (since $a+e+i=3$).
Next, the $(3,1)$-entry yields $a=1-h$, so $e=1+h$.
Substituting back and comparing the remaining entries gives
\[
c=-2h,\qquad b=-h^2,
\]
and therefore
\begin{equation}\label{eq:phi-x-1-final}
\varphi (x_{-\alpha}(1))=
\begin{pmatrix}
1-h& -h^2& -2h\\
1& 1+h& 2\\
1& h& 1
\end{pmatrix}.
\end{equation}

\subsection*{Step 6: A conjugacy reduction and an obstruction}
Note the factorization
\[
\varphi (x_{-\alpha}(1))=
\begin{pmatrix}
1& -h& 0\\
0& 1& 0\\
0& 0& 1
\end{pmatrix}
x_{-\alpha}(1)
\begin{pmatrix}
1& h& 0\\
0& 1& 0\\
0& 0& 1
\end{pmatrix}.
\]
Set $P=\begin{pmatrix}1&h&0\\0&1&0\\0&0&1\end{pmatrix}$.
Then $\varphi(x_{-\alpha}(-1))=P^{-1}x_{-\alpha}(-1)P$ in $\GL_3(R)$.
Since $\varphi(x_{-\alpha}(-1))$ is conjugate to $x_{-\alpha}(-1)$ inside $E(\mathbf A_1)$,
there exists $A\in E(\mathbf A_1)$ such that
$A^{-1}x_{-\alpha}(-1)A=\varphi(x_{-\alpha}(-1))$.
Hence $C:=AP^{-1}$ centralizes $x_{-\alpha}(-1)$, and we can write
\[
A=C\,P,\qquad Cx_{-\alpha}(-1)=x_{-\alpha}(-1)C.
\]
Arguing as in Lemma~\ref{lem:cent}, one checks that the centralizer of $x_{-\alpha}(-1)$ in $\mathrm{M}_3(R)$
consists of matrices of the form
\[
C=\begin{pmatrix}
a& 0& 0\\
b& a& 2c\\
c& 0& a
\end{pmatrix}
=a\begin{pmatrix}
1& 0& 0\\
d& 1& 0\\
0& 0& 1
\end{pmatrix}x_{-\alpha}(c/a)
\qquad(a\in R^*).
\]
Therefore, up to the factor $x_{-\alpha}(c/a)\in E(\mathbf A_1)$, we arrive at
\[
B:=a\begin{pmatrix}
1& h& 0\\
0& 1& 0\\
0& 0& 1
\end{pmatrix}
\begin{pmatrix}
1& 0& 0\\
d& 1& 0\\
0& 0& 1
\end{pmatrix}
=
\begin{pmatrix}
a(1+hd)& ah& 0\\
ad& a& 0\\
0& 0& a
\end{pmatrix}\in E(\mathbf A_1)\ \text{or }\ G(\mathbf A_1).
\]

\begin{lem}\label{lem:impossible}
The situation
\[
B=
\begin{pmatrix}
a(1+hd)& ah& 0\\
ad& a& 0\\
0& 0& a
\end{pmatrix}\in G(\mathbf A_1)
\]
is impossible for $h,d\neq 0$ and $a\neq 1$.
\end{lem}

\begin{proof}
Embed $R$ into the product of its localizations $S=\prod_{\mathfrak m} R_{\mathfrak m}$.
It suffices to show that for each maximal ideal $\mathfrak m$, the image $B_{\mathfrak m}\in G(\mathbf A_1,R_{\mathfrak m})$
forces $h_{\mathfrak m}=d_{\mathfrak m}=0$ and $a_{\mathfrak m}=1$.

Fix $\mathfrak m$ and write $R_{\mathfrak m}$ as a local ring.
Using the Gauss decomposition $G=TUVU$ for $\mathbf A_1$, we may write
\[
B_{\mathfrak m}=
\begin{pmatrix}
t& 0& 0\\
0& t^{-1}& 0\\
0& 0& 1
\end{pmatrix}
x_\alpha(\alpha)\,x_{-\alpha}(\beta)\,x_\alpha(\gamma)
\]
with $t\in R_{\mathfrak m}^*$ and $\alpha,\beta,\gamma\in R_{\mathfrak m}$.
Multiplying these factors explicitly gives (compare the $(3,1)$ and $(2,3)$ entries)
\[
\beta(\alpha\beta+1)=0,\qquad \beta(\beta\gamma+1)=0.
\]
If $\beta\in R_{\mathfrak m}^*$, then $\alpha\beta+1=\beta\gamma+1=0$, and the resulting matrix has a zero
on the diagonal positions $(2,2)$ and $(1,1)$, which is incompatible with the block form of $B_{\mathfrak m}$.
Hence $\beta$ is not a unit, i.e. $\beta\in\mathfrak mR_{\mathfrak m}$.
But then $\alpha\beta\in\mathfrak mR_{\mathfrak m}$, so $\alpha\beta+1\in 1+\mathfrak mR_{\mathfrak m}\subseteq R_{\mathfrak m}^*$,
and similarly $\beta\gamma+1\in R_{\mathfrak m}^*$.
Therefore the above equalities force $\beta=0$.

With $\beta=0$ we have $x_{-\alpha}(\beta)=1$, hence $B_{\mathfrak m}\in TU$, and in particular the $(2,1)$-entry must vanish.
Thus $ad=0$ in $R_{\mathfrak m}$, and since $a\in R^*$ in the group, we get $d=0$.
Then also $B_{\mathfrak m}$ becomes upper triangular, forcing $h=0$, and finally comparing the $(3,3)$-entry gives $a=1$.
\end{proof}

By Lemma~\ref{lem:impossible}, the only possibility in \eqref{eq:phi-x-1-final} is $h=0$, hence
\[
\varphi(x_{-\alpha}(1))=x_{-\alpha}(1).
\]

\subsection*{\texorpdfstring{Step 7: Forcing the diagonal parameter $\alpha=1$}{Step 7: Forcing the diagonal parameter alpha=1}}
From \eqref{eq:phi-xa} and \eqref{eq:beta-gamma} we know that
\[
\varphi(x_\alpha(t))=
\begin{pmatrix}
\alpha& t^2& 2t\\
0& \alpha& 0\\
0& t& \alpha
\end{pmatrix}\in E(\mathbf A_1)\qquad(t\in R),
\]
with $3\alpha=3$. We claim that necessarily $\alpha=1$.
Indeed, consider the general matrix
\[
X=\begin{pmatrix}
a& b^2& 2b\\
0& a& 0\\
0& b& a
\end{pmatrix}\in G(\mathbf A_1,R).
\]
Localizing at $R_{\mathfrak m}$ and repeating the Gauss-decomposition argument as in Lemma~\ref{lem:impossible},
comparison of the $(2,1)$-entry forces the negative-root parameter $\beta=0$, so $X\in TU$,
hence its $(2,2)$- and $(3,3)$-entries must be $1$. Therefore $a=1$.
Applying this to $X=\varphi(x_\alpha(t))$ yields $\alpha=1$.

\subsection*{\texorpdfstring{Finalization of the $\mathbf A_1$ case}{Finalization of the A1 case}}
We have shown that for all $t\in R$,
\[
\varphi(x_\alpha(t))=x_\alpha(t),\qquad \varphi(x_{-\alpha}(t))=x_{-\alpha}(t).
\]
Since $E(\mathbf A_1)$ is generated by these elements, the normalized endomorphism is the identity:
\[
\varphi=\mathrm{id}\quad\text{on }E(\mathbf A_1).
\]
Undoing the normalization \eqref{eq:norm} shows that the \emph{original} locally inner endomorphism is inner.

\begin{thm}\label{thm:A1}
Let $R$ be a commutative ring with $1/2\in R$. Then any locally inner endomorphism of
$E_{\mathrm{ad}}(\mathbf A_1,R)$ is inner. Equivalently, $E_{\mathrm{ad}}(\mathbf A_1,R)$ has Property~E, hence is $\Sha$-rigid.
\end{thm}

\begin{rem}
For the full adjoint Chevalley group $G_{\mathrm{ad}}(\mathbf A_1,R)$ the same conclusion follows whenever
$E_{\mathrm{ad}}(\mathbf A_1,R)\trianglelefteq G_{\mathrm{ad}}(\mathbf A_1,R)$.
\end{rem}

\begin{proof}
Let $\varphi:G_{\mathrm{ad}}\to G_{\mathrm{ad}}$ be locally inner. Restricting to $E_{\mathrm{ad}}$ and using Theorem~\ref{thm:A1},
we may assume $\varphi|_{E_{\mathrm{ad}}}=\mathrm{id}$.
For $g\in G_{\mathrm{ad}}$ and $x\in E_{\mathrm{ad}}$ we have $gxg^{-1}\in E_{\mathrm{ad}}$, hence
\[
gxg^{-1}= \varphi(gxg^{-1})= \varphi(g)x\varphi(g)^{-1}.
\]
Thus $g^{-1}\varphi(g)$ centralizes $E_{\mathrm{ad}}$. Since the centralizer of $E_{\mathrm{ad}}$ in $G_{\mathrm{ad}}$ is trivial,
we get $g^{-1}\varphi(g)=1$, hence $\varphi(g)=g$ for all $g$, i.e. $\varphi=\Id$.
\end{proof}

\medskip
\textit{Theorem~\ref{thm:main} is proved for the case $\mathbf A_1$.}

\section{\texorpdfstring{$\Sha$-Rigidity of the adjoint Chevalley groups of type~$\mathbf A_2$}{Sha-Rigidity of the adjoint Chevalley groups of type A2}}\label{sec:A2}

\subsection{\texorpdfstring{Basic notation and standard relations in type $\mathbf A_2$}{Basic notation and standard relations in type A2}}

Let $\Phi=\mathbf A_2$ with simple roots $\alpha_1,\alpha_2$ and $\gamma=\alpha_1+\alpha_2$.
We use the standard realization of $G_{\mathrm{ad}}(\mathbf A_2,R)\cong \mathrm{PGL}_3(R)$ coming from $\SL_3$:
\[
x_{\alpha_1}(t)=I+tE_{12},\quad
x_{\alpha_2}(t)=I+tE_{23},\quad
x_{\gamma}(t)=I+tE_{13},
\]
\[
x_{-\alpha_1}(t)=I+tE_{21},\quad
x_{-\alpha_2}(t)=I+tE_{32},\quad
x_{-\gamma}(t)=I+tE_{31}.
\]
For $u\in R^\times$ set
\[
h_{\alpha_1}(u)=\diag(u,u^{-1},1),\qquad
h_{\alpha_2}(u)=\diag(1,u,u^{-1}),
\]
and define $w_{\alpha_i}(u)=x_{\alpha_i}(u)x_{-\alpha_i}(-u^{-1})x_{\alpha_i}(u)$.

We will use the following relations (all other commutators between root subgroups in type $\mathbf A_2$
are trivial since $r+s\notin\Phi$):
\begin{align*}
x_r(t)\,x_r(s)&=x_r(t+s)\qquad (r\in\Phi),\\
[x_{\alpha_1}(t),x_{\alpha_2}(s)]&=x_{\gamma}(ts),\\
x_{\alpha_2}(v)\,x_{\alpha_1}(u)&=
x_{\alpha_1}(u)\,x_{\alpha_2}(v)\,x_{\gamma}(-uv),\\
[x_{\alpha_1}(t),x_{-\gamma}(s)]&=x_{-\alpha_2}(-ts),\qquad
[x_{\alpha_2}(t),x_{-\gamma}(s)]=x_{-\alpha_1}(ts),\\
[x_{\gamma}(t),x_{-\alpha_1}(s)]&=x_{\alpha_2}(-ts),\qquad
[x_{\gamma}(t),x_{-\alpha_2}(s)]=x_{\alpha_1}(ts).
\end{align*}
Moreover $w_{\alpha_i}(u)^2=h_{\alpha_i}(-1)$ for $i=1,2$.

\subsection{\texorpdfstring{The special element $X_0$ and its centralizer}{The special element X0 and its centralizer}}\label{subsec:X0}

Let $X_0=x_{\alpha_1}(1)\,x_{\alpha_2}(1)$ and assume that $\varphi(X_0)=X_0$ for our initial $\varphi$.

We determine the centralizer of $X_0$ in $G_{\mathrm{ad}}(\mathbf A_2,R)$ (equivalently, in $E_{\mathrm{ad}}(\mathbf A_2,R)$).
If $g$ commutes with $X_0$, then for every maximal ideal $J\subset R$ its localization $g_J$ commutes with $(X_0)_J$.
Hence it suffices to find the centralizer in a local ring $R_J$; reducing further modulo $J$, the image $\overline g$ commutes with
$\overline X_0$ over the residue field~$k_J$.

\medskip
\noindent\textbf{Field step.}
Over a field we may write
\[
\overline g \;=\; t_1(a_1)t_2(a_2)\;x_{\alpha_1}(b_1)x_{\alpha_2}(b_2)x_{\alpha_1+\alpha_2}(b_3)\;\mathbf w\;
x_{\alpha_1}(c_1)x_{\alpha_2}(c_2)x_{\alpha_1+\alpha_2}(c_3),
\]
and require
\[
\overline g\,x_{\alpha_1}(1)x_{\alpha_2}(1)\,\overline g^{-1}\;=\;x_{\alpha_1}(1)x_{\alpha_2}(1).
\]
Then
\begin{multline*}
x_{\alpha_1}(1)x_{\alpha_2}(1)
= \\
=g'\,x_{\alpha_1}(c_1)x_{\alpha_2}(c_2)x_{\alpha_1+\alpha_2}(c_3)\,x_{\alpha_1}(1)x_{\alpha_2}(1)\,
\bigl(x_{\alpha_1}(c_1)x_{\alpha_2}(c_2)x_{\alpha_1+\alpha_2}(c_3)\bigr)^{-1}\,(g')^{-1}\\
= g'\,x_{\alpha_1}(1)x_{\alpha_2}(1)\,x_{\alpha_1+\alpha_2}(d)\,(g')^{-1}
= g''\,\mathbf w\,x_{\alpha_1}(1)x_{\alpha_2}(1)\,x_{\alpha_1+\alpha_2}(d)\,\mathbf w^{-1}\,(g'')^{-1}\\
= g''\,x_{w(\alpha_1)}(1)x_{w(\alpha_2)}(1)\,x_{w(\alpha_1+\alpha_2)}(d)\,(g'')^{-1}.
\end{multline*}
The left-hand side contains only positive root factors, so necessarily $\mathbf w=e$ (otherwise a negative root would appear).
Next,
\begin{multline*}
x_{w(\alpha_1)}(1)x_{w(\alpha_2)}(1)x_{w(\alpha_1+\alpha_2)}(d)
=\\
=g'''\,t_1(a_1)^{-1}t_2(a_2)^{-1}\,x_{\alpha_1}(1)x_{\alpha_2}(1)\,t_1(a_1)t_2(a_2)\,(g''')^{-1}\\
= g'''\,x_{\alpha_1}(a_1^{-1})x_{\alpha_2}(a_2^{-1})\,(g''')^{-1}\\
= x_{\alpha_1}(-b_1)x_{\alpha_2}(-b_2)x_{\alpha_1+\alpha_2}(-b_3)\,x_{\alpha_1}(a_1^{-1})x_{\alpha_2}(a_2^{-1})
x_{\alpha_1}(b_1)x_{\alpha_2}(b_2)x_{\alpha_1+\alpha_2}(b_3)\\
= x_{\alpha_1}(a_1^{-1})x_{\alpha_2}(a_2^{-1})\,x_{\alpha_1+\alpha_2}(d').
\end{multline*}
Thus $a_1=a_2=1$, hence $t_1=t_2=1$, and
\[
\overline g \;=\; x_{\alpha_1}(b_1)\,x_{\alpha_2}(b_2)\,x_{\alpha_1+\alpha_2}(b_3).
\]
Finally,
$$
x_{\alpha_1}(b_1)x_{\alpha_2}(b_2)\,x_{\alpha_1}(1)x_{\alpha_2}(1)\,x_{\alpha_1}(-b_1)x_{\alpha_2}(-b_2)
= x_{\alpha_1}(1)x_{\alpha_2}(1)\,x_{\alpha_1+\alpha_2}(b_1-b_2),
$$
so $b_1=b_2$, while $b_3$ is arbitrary.

\medskip
\noindent\textbf{Local step.}
Let $R$ be local with maximal ideal $J$ and residue field $k=R/J$.
If $g$ commutes with $X_0$ in $G_{\mathrm{ad}}(\mathbf A_2,R)$, then $\overline g$ commutes with $\overline X_0$ in $G_{\mathrm{ad}}(\mathbf A_2,k)$ and hence
\[
\overline g=\overline{x_{\alpha_1}(b)}\,\overline{x_{\alpha_2}(b)}\,\overline{x_{\alpha_1+\alpha_2}(b_3)}.
\]
Therefore $g$ admits the short Gauss factorization
\[
g=x_{-\alpha_1}(u_1)\,x_{-\alpha_2}(u_2)\,x_{-\alpha_1-\alpha_2}(u_3)\,t_1(r_1)t_2(r_2)\,
x_{\alpha_1}(s_1)\,x_{\alpha_2}(s_2)\,x_{\alpha_1+\alpha_2}(s_3),
\]
with $u_i\in J$, $r_i\equiv 1\pmod J$, and $s_1\equiv s_2\pmod J$.
From $gX_0g^{-1}=X_0$ we compute
\begin{multline*}
x_{\alpha_1}(1)x_{\alpha_2}(1)
=\\
=g'\,x_{\alpha_1}(s_1)x_{\alpha_2}(s_2)x_{\alpha_1+\alpha_2}(s_3)\,x_{\alpha_1}(1)x_{\alpha_2}(1)\,
\bigl(x_{\alpha_1}(s_1)x_{\alpha_2}(s_2)x_{\alpha_1+\alpha_2}(s_3)\bigr)^{-1}\,(g')^{-1}=\\
= g'\,x_{\alpha_1}(1)x_{\alpha_2}(1)\,x_{\alpha_1+\alpha_2}(d)\,(g')^{-1}
=\\
=g''\,t_1(r_1)t_2(r_2)\,x_{\alpha_1}(1)x_{\alpha_2}(1)\,x_{\alpha_1+\alpha_2}(d)\,t_1(r_1^{-1})t_2(r_2^{-1})\,(g'')^{-1}\\
= g''\,x_{\alpha_1}(r_1)\,x_{\alpha_2}(r_2)\,x_{\alpha_1+\alpha_2}(d')\,(g'')^{-1}\\
= x_{-\alpha_1}(u_1)x_{-\alpha_2}(u_2)x_{-\alpha_1-\alpha_2}(u_3)\,x_{\alpha_1}(r_1)x_{\alpha_2}(r_2)x_{\alpha_1+\alpha_2}(d')\times\\
\times
x_{-\alpha_1}(-u_1)x_{-\alpha_2}(-u_2)x_{-\alpha_1-\alpha_2}(-u_3).
\end{multline*}
Since $r_1,r_2\in R^\times$, the right-hand side would contain negative root factors whenever $u_3\neq 0$, while the left-hand side does not; hence $u_3=0$.
Moreover, commuting $x_{-\alpha_i}(u_i)$ past $x_{\alpha_i}(r_i)$ (rank-one relation) produces a nontrivial torus factor $h_{\alpha_i}(\cdots)$
whenever $u_i\neq 0$, which cannot be absorbed because $r_i$ are units and no torus appears on the left;
therefore $u_1=u_2=0$.
Thus
\[
g=t_1(r_1)t_2(r_2)\,x_{\alpha_1}(s_1)x_{\alpha_2}(s_2)x_{\alpha_1+\alpha_2}(s_3).
\]
Finally, $gX_0g^{-1}=X_0$ forces $t_1=t_2=1$ (a nontrivial torus rescales $x_{\alpha_i}(1)$), and the field-step computation above gives $s_1=s_2$.
Hence for some $a,b\in R$,
\[
g=x_{\alpha_1}(a)\,x_{\alpha_2}(a)\,x_{\alpha_1+\alpha_2}(b).
\]

\medskip
\noindent\textbf{Global step.}
For general $R$, $g$ centralizes $X_0$ if and only if each localization $g_J$ centralizes $(X_0)_J$ for all maximal $J$.
By the local description,
$$
g_J=x_{\alpha_1}(a_J)x_{\alpha_2}(a_J)x_{\alpha_1+\alpha_2}(b_J).
$$
By patching, there exist $a,b\in R$ with these local images. The converse inclusion is immediate from the displayed computations.

\begin{prop}\label{prop:CentX0}
\[
C_{G_{\mathrm{ad}}(\mathbf A_2,R)}(X_0)
=\Bigl\{\,x_{\alpha_1}(a)\,x_{\alpha_2}(a)\,x_{\alpha_1+\alpha_2}(b)\ \Bigm|\ a,b\in R\,\Bigr\}.
\]
\end{prop}

\subsection{\texorpdfstring{Image of $x_{\alpha_1+\alpha_2}(1)$}{Image of xalpha1+alpha2(1)}}\label{subsec:image-a1+a2}

Let $X_{12}:=\varphi\bigl(x_{\alpha_1+\alpha_2}(1)\bigr)$.
Since $\varphi$ is locally inner, $X_{12}$ is conjugate in the Chevalley group to $x_{\alpha_1+\alpha_2}(1)$.
Moreover, by Proposition~\ref{prop:CentX0}, $X_{12}$ commutes with $X_0$, hence
\[
X_{12}=x_{\alpha_1}(a)\,x_{\alpha_2}(a)\,x_{\alpha_1+\alpha_2}(b)\qquad\text{for some }a,b\in R,
\]
and for every maximal ideal $J\subset R$ its localization $(X_{12})_J$ is conjugate to $\bigl(x_{\alpha_1+\alpha_2}(1)\bigr)_J$ in $G_{\mathrm{ad}}(\mathbf A_2,R_J)$.

\medskip
\noindent\textbf{Field step.}
Fix $J$ and pass to the residue field $k_J=R_J/\Rad R_J$.
Since $x_{\alpha_1+\alpha_2}(1)$ commutes with all positive root unipotents, we may assume the conjugating element over $k_J$ has the form
\[
g=t_1(a_1)\,t_2(a_2)\,x_{\alpha_1}(b_1)\,x_{\alpha_2}(b_2)\,x_{\alpha_1+\alpha_2}(b_3)\,\mathbf w,
\]
so that
\begin{multline*}
(x_{\alpha_1}(b_1)x_{\alpha_2}(b_2)x_{\alpha_1+\alpha_2}(b_3))^{-1}\,
x_{\alpha_1}(a_1^{-1}a)\,x_{\alpha_2}(a_2^{-1}a)\,x_{\alpha_1+\alpha_2}(a_1^{-1}a_2^{-1}b)\times \\
\times x_{\alpha_1}(b_1)x_{\alpha_2}(b_2)x_{\alpha_1+\alpha_2}(b_3)
= x_{w(\alpha_1+\alpha_2)}(1).
\end{multline*}
Pushing $x_{\alpha_1}(b_1)$ and $x_{\alpha_2}(b_2)$ through gives
\[
x_{\alpha_1}(a_1^{-1}a)\,x_{\alpha_2}(a_2^{-1}a)\,
x_{\alpha_1+\alpha_2}\bigl(-b_2a_1^{-1}a-b_1a_2^{-1}a+a_1^{-1}a_2^{-1}b\bigr)
= x_{w(\alpha_1+\alpha_2)}(1).
\]
If $a\neq 0$ in $k_J$, then the left-hand side contains nontrivial $x_{\alpha_1}$ and $x_{\alpha_2}$ factors, which is impossible for a single root element on the right.
Hence $a\equiv 0\pmod{J}$ for every~$J$, i.e. $a\in \Rad(R)$.
Then we see $w(\alpha_1+\alpha_2)=\alpha_1+\alpha_2$, which necessarily implies $\mathbf w=e$.

\medskip
\noindent\textbf{Local step.}
Fix a maximal ideal $J$ and work over the local ring $R_J$.
As above, we may take the conjugating element in the form
\[
g=x_{-\alpha_1}(c_1)\,x_{-\alpha_2}(c_2)\,x_{-\alpha_1-\alpha_2}(c_3)\,t_1(r_1)\,t_2(r_2),
\qquad c_i\in \Rad(R_J),\ r_i\in R_J^\times.
\]
Conjugating $x_{\alpha_1+\alpha_2}(r_1r_2)$ by $x_{-\alpha_1-\alpha_2}(c_3)$ produces the torus factor
$h_{\alpha_1+\alpha_2}(1+c_3)$; since
\[
x_{\alpha_1}(a)\,x_{\alpha_2}(a)\,x_{\alpha_1+\alpha_2}(b)
\]
has no torus part, we must have $c_3=0$.
Hence
\begin{multline*}
x_{\alpha_1}(a)x_{\alpha_2}(a)x_{\alpha_1+\alpha_2}(b)
=x_{-\alpha_1}(c_1)\,x_{-\alpha_2}(c_2)\,x_{\alpha_1+\alpha_2}(r_1r_2)\,x_{-\alpha_2}(-c_2)\,x_{-\alpha_1}(-c_1)=\\
=x_{-\alpha_1}(c_1)\,x_{\alpha_1+\alpha_2}(r_1r_2)\,x_{\alpha_1}(\pm c_2 r_1r_2)\,x_{-\alpha_1}(-c_1)=\\
=x_{\alpha_1+\alpha_2}(r_1r_2)\,x_{\alpha_2}(\pm c_1 r_1r_2)\,h_{\alpha_1}(1\pm c_1c_2r_1r_2)\,x_{\alpha_1}(\pm c_2 r_1r_2)\,x_{-\alpha_1}(\dots).
\end{multline*}
Since the left-hand side lies in the positive unipotent group $U$ (no torus, no negative root factors), we must have
\[
h_{\alpha_1}(1\pm c_1c_2r_1r_2)=1\quad\Longrightarrow\quad c_1c_2=0,
\]
and the negative factor $x_{-\alpha_1}(\dots)$ must vanish in the product, which forces
\[
c_1=\pm c_2,\qquad c_1^2=0.
\]
Comparing the remaining positive factors with $x_{\alpha_1}(a)\,x_{\alpha_2}(a)\,x_{\alpha_1+\alpha_2}(b)$ yields
\[
b=r_1r_2\in R_J^\times,\qquad a=\pm c_2\,r_1r_2,
\]
and therefore $a^2=(c_2^2)\,(r_1r_2)^2=0$ in $R_J$.

\medskip
\noindent\textbf{Global step.}
We have shown for every localization $R_J$ that $b$ maps to a unit and $a^2$ maps to $0$.
Hence $b\in R^\times$ and $a^2=0$ in $R$.
Summarizing:

\begin{prop}\label{prop:image-a1+a2}
Let $X_{12}=\varphi\bigl(x_{\alpha_1+\alpha_2}(1)\bigr)$.
Then
\[
X_{12}=x_{\alpha_1}(a)\,x_{\alpha_2}(a)\,x_{\alpha_1+\alpha_2}(b)
\quad\text{with}\quad a^2=0\ \ \text{and}\ \ b\in R^\times.
\]
\end{prop}

\subsection{\texorpdfstring{The image of $x_{\alpha_1}(1)$}{The image of xalpha1(1)}}\label{subsec:image-a1}

Let $X_1:=\varphi(x_{\alpha_1}(1))$. Then $X_1$ commutes with $X_{12}$ in all localizations and, hence, after reduction to every residue field.
Recall from Subsection~\ref{subsec:image-a1+a2} that
\[
X_{12}=x_{\alpha_1}(a)\,x_{\alpha_2}(a)\,x_{\alpha_1+\alpha_2}(b),\qquad a^2=0,\quad b\in R^\times,
\]
and in particular $\overline a=0$ in each residue field.

\medskip
\noindent\textbf{Field step.}
Let $k$ be a residue field of $R$.
We already know that in this case $a=0$, so
\[
X_{12}=x_{\alpha_1+\alpha_2}(b),\qquad b\in k^\times.
\]
Since $X_1$ commutes with $X_{12}$, it normalizes the one-parameter subgroup
$\{x_{\alpha_1+\alpha_2}(t)\mid t\in k\}$.

Write the Bruhat decomposition of $X_1$:
\[
X_1=t_1(a_1)t_2(a_2)\,x_{\alpha_1}(b_1)\,x_{\alpha_2}(b_2)\,
x_{\alpha_1+\alpha_2}(b_3)\,\mathbf w\,
x_{\alpha_1}(c_1)\,x_{\alpha_2}(c_2)\,x_{\alpha_1+\alpha_2}(c_3),
\]
where $\mathbf w$ is an element of the Weyl group of type~$\mathbf A_2$.

Conjugating $X_{12}=x_{\alpha_1+\alpha_2}(b)$ by this element gives
\begin{multline*}
X_1\,x_{\alpha_1+\alpha_2}(b)\,X_1^{-1}
=t_1(a_1)t_2(a_2)\,x_{\alpha_1}(b_1)\,x_{\alpha_2}(b_2)\,
x_{\alpha_1+\alpha_2}(b_3)\,\mathbf w\,
x_{\alpha_1+\alpha_2}(b)\,\mathbf w^{-1}\times \\
\times
\bigl(x_{\alpha_1}(b_1)x_{\alpha_2}(b_2)x_{\alpha_1+\alpha_2}(b_3)\bigr)^{-1}
t_1(a_1)^{-1}t_2(a_2)^{-1}.
\end{multline*}

Since conjugation by $t_i(a_i)$ acts by the root weights,
\[
t_1(a_1)t_2(a_2)\,x_{\alpha_1+\alpha_2}(b)\,t_1(a_1)^{-1}t_2(a_2)^{-1}
= x_{\alpha_1+\alpha_2}(a_1^{-1}a_2^{-1}b),
\]
and conjugation by $\mathbf w$ sends $x_{\alpha_1+\alpha_2}(b)$ to
$x_{w(\alpha_1+\alpha_2)}(\pm b)$.
Thus the condition $X_1x_{\alpha_1+\alpha_2}(b)X_1^{-1}=x_{\alpha_1+\alpha_2}(b)$
implies
\[
x_{\alpha_1+\alpha_2}(a_1^{-1}a_2^{-1}b)
= x_{w(\alpha_1+\alpha_2)}(b).
\]
Now, if $\mathbf w\neq e$, then $w(\alpha_1+\alpha_2)$ is a negative root,
while the left side is a positive root element, which is impossible.
Hence $\mathbf w=e$.
With $\mathbf w=e$, both sides correspond to the same positive root, and
equality of one-parameter subgroups gives $a_1^{-1}a_2^{-1}=1$,
i.e. $a_1a_2=1$.
Therefore in the field case $\mathbf w=e$, $a_1a_2=1$.

\medskip
\noindent\textbf{Local step.}
Fix a maximal ideal $J\subset R$ and work over the local ring $R_J$.
Write
\[
X_1=t_1(r_1)t_2(r_2)\,x_{\alpha_1}(s_1)\,x_{\alpha_2}(s_2)\,x_{\alpha_1+\alpha_2}(s_3)\,
x_{-\alpha_1}(u_1)\,x_{-\alpha_2}(u_2)\,x_{-\alpha_1-\alpha_2}(u_3),
\]
with $u_i\in \Rad R_J$.
We also keep the form
\[
X_{12}=x_{\alpha_1}(a)\,x_{\alpha_2}(a)\,x_{\alpha_1+\alpha_2}(b),\qquad a^2=0,\quad b\in R^\times.
\]
If $u_3\neq0$, conjugating $x_{\alpha_1+\alpha_2}(b)$ by $x_{-\alpha_1-\alpha_2}(u_3)$ produces a nontrivial torus factor
$h_{\alpha_1+\alpha_2}(1+u_3\cdot(\cdots))$ which cannot be absorbed; since $b$ is a unit and no such torus appears on the other side, we must have $u_3=0$.

The commutativity $X_{12}X_1=X_1X_{12}$ gives, after cancelling the common outer torus $t_1(r_1)t_2(r_2)$ and the positive unipotent block of $X_1$,
\begin{align*}
&x_{\alpha_1}(r_1a)\,x_{\alpha_2}(r_2a)\,x_{\alpha_1+\alpha_2}(r_1r_2b)\,x_{-\alpha_1}(u_1)\,x_{-\alpha_2}(u_2)\\
&\qquad=\ x_{-\alpha_1}(u_1)\,x_{-\alpha_2}(u_2)\,x_{\alpha_1}(a)\,x_{\alpha_2}(a)\,x_{\alpha_1+\alpha_2}(b).
\end{align*}
Moving $x_{-\alpha_i}(u_i)$ through $x_{\alpha_i}(a)$ (rank-one identity) creates the torus factor $h_{\alpha_i}(1+au_i)$.
There is no torus factor on the left, hence
\[
au_1=au_2=0. \tag{$\ast$}
\]
Reorder the right-hand side through $x_{\alpha_1+\alpha_2}(b)$ using the Chevalley relations; one gets
\[
x_{\alpha_1}(r_1a)\,x_{\alpha_2}(r_2a)\,x_{\alpha_1+\alpha_2}(r_1r_2b)
=\ x_{\alpha_1}(a\pm u_2 b)\,x_{\alpha_2}(a\pm u_1 b)\,x_{\alpha_1+\alpha_2}(b).
\]
Comparing the $\alpha_1$- and $\alpha_2$-coordinates gives
\[
r_1a=a\pm u_2 b,\qquad r_2a=a\pm u_1 b.
\]
Since $b\in R^\times$, this shows that
\[
u_2=\pm (r_1-1)\,b^{-1}\,a,\qquad u_1=\pm (r_2-1)\,b^{-1}\,a,
\]
i.e.
\[
u_1=a\gamma_1,\qquad u_2=a\gamma_2\qquad\text{for some }\gamma_1,\gamma_2\in R_J. \tag{$\dagger$}
\]
With \textup{($\dagger$)}, any correction to the $\alpha_1+\alpha_2$-coordinate coming from commuting $x_{-\alpha_i}(u_i)$ past $x_{\alpha_i}(a)$ is proportional to $a^2$ and hence vanishes.
Therefore comparing the $\alpha_1+\alpha_2$-coordinate gives
\[
r_1r_2\,b=b\quad\Longrightarrow\quad r_1r_2=1.
\]
Set $r:=r_1\in 1+J$ and write $t_1(r)t_2(r^{-1})$ for the outer torus.
Using \textup{($\dagger$)} we obtain the normalized local expression
\[
X_1
= t_1(r)\,t_2(r^{-1})\,
x_{\alpha_1}(s_1)\,x_{\alpha_2}(s_2)\,x_{\alpha_1+\alpha_2}(s_3)\,
x_{-\alpha_1}(a\gamma_1)\,x_{-\alpha_2}(a\gamma_2),
\]
with $a^2=0$ and $b\in R^\times$ as above.

\subsection{\texorpdfstring{Second constraint from the relation $x_{\alpha_1}(1)X_0=x_{\alpha_1+\alpha_2}(1)\,X_0\,x_{\alpha_1}(1)$}{Second constraint from the relation xalpha1(1)X0=xalpha1+alpha2(1)X0xalpha1(1)}}\label{subsec:second-constraint}

Applying $\varphi$ to this identity gives
\[
X_1\,X_0 = X_{12}\,X_0\,X_1.
\]
Recall that in a localization $R_J$ we have
\[
X_{12}=x_{\alpha_1}(a)\,x_{\alpha_2}(a)\,x_{\alpha_1+\alpha_2}(b),
\qquad a^2=0,\quad b\in R_J^\times,
\]
and
\[
X_1=t_1(r)\,t_2(r^{-1})\,
x_{\alpha_1}(s_1)\,x_{\alpha_2}(s_2)\,x_{\alpha_1+\alpha_2}(s_3)\,
x_{-\alpha_1}(a\gamma_1)\,x_{-\alpha_2}(a\gamma_2).
\]
In the special case $u^2=0$ one has
\begin{equation}\label{eq:nilpotent-rank1}
x_{-\alpha}(u)\,x_{\alpha}(1)
= h_{\alpha}(1-u)\,x_{\alpha}(1+u)\,x_{-\alpha}(u).
\end{equation}
Applying \eqref{eq:nilpotent-rank1} to $x_{-\alpha_i}(a\gamma_i)\,x_{\alpha_i}(1)$ moves the negative root factors to the far right. Hence, for comparing the $T$- and $U^+$-parts in the identity $X_1X_0=X_{12}X_0X_1$, we may ignore the common rightmost block and work with
\[
\widetilde X_1:=t_1(r)\,t_2(r^{-1})\,x_{\alpha_1}(s_1)\,x_{\alpha_2}(s_2)\,x_{\alpha_1+\alpha_2}(s_3).
\]
Conjugating $X_0$ past the torus gives
\[
X_0\,t_1(r)\,t_2(r^{-1})
= t_1(r)\,t_2(r^{-1})\,x_{\alpha_1}(r^{-1})\,x_{\alpha_2}(r).
\]
Multiplying $X_1X_0=X_{12}X_0X_1$ on the left by $t_1(r)^{-1}t_2(r)$
we obtain an equality entirely inside the positive unipotent subgroup:
\begin{multline}\label{eq:Uplus-eq}
x_{\alpha_1}(s_1)\,x_{\alpha_2}(s_2)\,x_{\alpha_1+\alpha_2}(s_3)\,x_{\alpha_1}(1)\,x_{\alpha_2}(1)
=\\
=x_{\alpha_1}(a)\,x_{\alpha_2}(a)\,x_{\alpha_1+\alpha_2}(b)\,
x_{\alpha_1}(r^{-1})\,x_{\alpha_2}(r)\,
x_{\alpha_1}(s_1)\,x_{\alpha_2}(s_2)\,x_{\alpha_1+\alpha_2}(s_3).
\end{multline}
Using the standard order $\alpha_1,\alpha_2,\alpha_1+\alpha_2$ and the relation
$$
x_{\alpha_2}(v)\,x_{\alpha_1}(u)
= x_{\alpha_1}(u)\,x_{\alpha_2}(v)\,x_{\alpha_1+\alpha_2}(- uv),
$$
the $\alpha_1$- and $\alpha_2$-parameters in \eqref{eq:Uplus-eq} are
\[
s_1+1,\ s_2+1\quad\text{on the left, and}\quad r^{-1}+s_1,\ r+s_2\quad\text{on the right.}
\]
Hence $r=1$.
For $r=1$ the only contribution to the $\alpha_1+\alpha_2$-coordinate comes from commuting $x_{\alpha_2}(s_2)$ past $x_{\alpha_1}(1)$.
Comparing these parameters yields
\[
b=s_1-s_2,
\]
which is a unit in $R_J$ since $b\in R_J^\times$.
Conjugating $X_1$ by $x_{\alpha_1}(\mu)x_{\alpha_2}(\mu)$ affects only the $\alpha_1+\alpha_2$-coordinate:
\begin{multline*}
x_{\alpha_1}(\mu)x_{\alpha_2}(\mu)
\,x_{\alpha_1}(s_1)x_{\alpha_2}(s_2)x_{\alpha_1+\alpha_2}(s_3)\,
x_{\alpha_2}(-\mu)x_{\alpha_1}(-\mu)
=\\
=x_{\alpha_1}(s_1)x_{\alpha_2}(s_2)
  x_{\alpha_1+\alpha_2}\bigl(s_3+(s_2-s_1)\mu\bigr).
\end{multline*}
Since $b=s_1-s_2$ is invertible, choosing $\mu=s_3/b$ eliminates the $\alpha_1+\alpha_2$-component.
Renaming $s:=s_2$ (so $s_1=b+s$) gives the local normal form
\[
X_1=x_{\alpha_1}(b+s)\,x_{\alpha_2}(s),
\qquad  b\in R_J^\times.
\]

\subsection{\texorpdfstring{Field conjugacy constraints for $X_1$ and $X_2$}{Field conjugacy constraints for X1 and X2}}\label{subsec:field-conjugacy}

Let $J\subset R$ be a maximal ideal. Work in the localization $R_J$ with residue field
$k_J:=R_J/\Rad R_J$, and denote reduction modulo $\Rad R_J$ by a bar.
From the previous subsection we have the local normal form
\begin{equation}\label{eq:X1-local}
X_1 = x_{\alpha_1}(b+s)\,x_{\alpha_2}(s),
\qquad b\in R_J^\times,\ s\in R_J.
\end{equation}
Since $X_0=x_{\alpha_1}(1)x_{\alpha_2}(1)$, in the standard $A_2$ realization one checks
\[
X_0^{-1}x_{\alpha_1}(1)=x_{\alpha_2}(-1),
\]
hence
\begin{equation}\label{eq:X2inv-local}
X_2^{-1}=\varphi(x_{\alpha_2}(-1))=X_0^{-1}X_1
= x_{\alpha_1}(b+s-1)\,x_{\alpha_2}(s-1)\,x_{\alpha_1+\alpha_2}(1-b-s).
\end{equation}
(Equivalently, in $U$ one has
\begin{equation}\label{eq:X2-local}
X_2
= x_{\alpha_1}(1-b-s)\,x_{\alpha_2}(1-s)\,x_{\alpha_1+\alpha_2}\bigl(s(s+b-1)\bigr),
\end{equation}
obtained by inverting \eqref{eq:X2inv-local} inside the group~$U$.)

\medskip
\noindent\textbf{A field criterion.}
Over the field $k_J$ we identify $G_{\mathrm{ad}}(\mathbf A_2,k_J)\simeq \mathrm{PGL}_3(k_J)$ using the
standard root subgroups. In this model,
a root element satisfies
\[
(g-I)^2=0\quad\text{and}\quad g\neq I.
\]
Since $\varphi$ preserves conjugacy classes, the same nilpotency condition must hold for
$\overline{X_1}$ and $\overline{X_2}$.
Write a general element of $U(k_J)$ in the ordered form
\[
u=x_{\alpha_1}(u_1)\,x_{\alpha_2}(u_2)\,x_{\alpha_1+\alpha_2}(u_3).
\]
A direct matrix multiplication gives
\[
u = I + u_1E_{12}+u_2E_{23}+(u_3+u_1u_2)E_{13},
\qquad\text{hence}\qquad
(u-I)^2 = (u_1u_2)E_{13}.
\]
Therefore
\begin{equation}\label{eq:transvection-criterion}
(u-I)^2=0 \iff u_1u_2=0.
\end{equation}
From \eqref{eq:X1-local} we have
\[
\overline{X_1}=x_{\alpha_1}(\overline{b}+\overline{s})\,x_{\alpha_2}(\overline{s}).
\]
Applying \eqref{eq:transvection-criterion} yields
\[
\overline{s}\,(\overline{b}+\overline{s})=0.
\]
Since $\overline{b}\neq 0$ in $k_J$, there are only two possibilities:
\[
\text{either }\ \overline{s}=0,\qquad
\text{or }\ \overline{b}+\overline{s}=0.
\]
From \eqref{eq:X2inv-local},
\[
\overline{X_2^{-1}}
= x_{\alpha_1}(\overline{b}+\overline{s}-1)\,x_{\alpha_2}(\overline{s}-1)\,
  x_{\alpha_1+\alpha_2}(1-\overline{b}-\overline{s}).
\]
Applying \eqref{eq:transvection-criterion} again gives
\[
(\overline{b}+\overline{s}-1)(\overline{s}-1)=0.
\]
Combining with the two alternatives above, we obtain exactly two mutually exclusive reductions:

\smallskip
\noindent\emph{Case I:} $\overline{s}=0$ forces $\overline{b}=1$.
Equivalently,
\[
s\in\Rad R_J,\qquad b\equiv 1\pmod J.
\]

\smallskip
\noindent\emph{Case II:} $\overline{b}+\overline{s}=0$ forces $\overline{s}=1$ and hence $\overline{b}=-1$.
Equivalently,
\[
s\equiv 1\pmod J,\qquad b\equiv -1\pmod J.
\]
Over $k_J$ this means
\[
\overline{X_1}=x_{\alpha_2}(1),\qquad
\overline{X_2}=x_{\alpha_1}(1)\,x_{\alpha_1+\alpha_2}(-1).
\]

\subsection{The images of diagonal involutions}\label{subsec:diag-invol}

The involution $h_{\alpha_1}(-1)$ commutes with $x_{\alpha_1}(1)$ and inverts $x_{\alpha_2}(1)$.
The involution $h_{\alpha_2}(-1)$ commutes with $x_{\alpha_2}(1)$ and inverts $x_{\alpha_1}(1)$.
They commute, and their product $h_{\alpha_1+\alpha_2}(-1)$ commutes with $x_{\alpha_1+\alpha_2}(1)$, inverts $x_{\alpha_1}(1)$ and $x_{\alpha_2}(1)$, and
\[
X_0^{h_{\alpha_1+\alpha_2}(-1)}
= x_{\alpha_1}(-1)x_{\alpha_2}(-1)
= x_{\alpha_2}(-1)x_{\alpha_1}(-1)x_{\alpha_1+\alpha_2}(1)
= X_0^{-1}\,x_{\alpha_1+\alpha_2}(1).
\]
Set
\[
H_1=\varphi\bigl(h_{\alpha_1}(-1)\bigr),\qquad
H_2=\varphi\bigl(h_{\alpha_2}(-1)\bigr),\qquad
H_{12}=\varphi\bigl(h_{\alpha_1+\alpha_2}(-1)\bigr).
\]

\noindent\textbf{The element $H_{12}$ modulo the radical.}
Since $H_{12}$ commutes with $X_{12}$, passing modulo the radical we obtain
\begin{multline*}
h_{\alpha_1}(a_1)h_{\alpha_2}(a_2)\,
x_{\alpha_1}(b_1)x_{\alpha_2}(b_2)x_{\alpha_1+\alpha_2}(b_3)\,
\mathbf w\,
x_{\alpha_1}(c_1)x_{\alpha_2}(c_2)x_{\alpha_1+\alpha_2}(c_3)\,
x_{\alpha_1+\alpha_2}(b)=\\
=\;
x_{\alpha_1+\alpha_2}(b)\,
h_{\alpha_1}(a_1)h_{\alpha_2}(a_2)\,
x_{\alpha_1}(b_1)x_{\alpha_2}(b_2)x_{\alpha_1+\alpha_2}(b_3)\,
\mathbf w\,
x_{\alpha_1}(c_1)x_{\alpha_2}(c_2)x_{\alpha_1+\alpha_2}(c_3).
\end{multline*}
This simplifies to
\[
x_{w(\alpha_1+\alpha_2)}(b)=x_{\alpha_1+\alpha_2}(a_1a_2\,b),
\]
hence $w(\alpha_1+\alpha_2)=\alpha_1+\alpha_2$ (so $\mathbf w=e$) and $a_1a_2=1$.
Using $H_{12}^2=1$, we have
\begin{multline*}
h_{\alpha_1}(a_1)h_{\alpha_2}(1/a_1)\,x_{\alpha_1}(b_1)x_{\alpha_2}(b_2)x_{\alpha_1+\alpha_2}(b_3)=\\
= x_{\alpha_1+\alpha_2}(-b_3)x_{\alpha_2}(-b_2)x_{\alpha_1}(-b_1)\,h_{\alpha_1}(1/a_1)h_{\alpha_2}(a_1)=\\
= h_{\alpha_1}(1/a_1)h_{\alpha_2}(a_1)\,x_{\alpha_2}(-b_2/a_1)\,x_{\alpha_1}(-a_1b_1)\,x_{\alpha_1+\alpha_2}(-b_3).
\end{multline*}
It follows that $a_1=\pm 1$ and $2b_3=b_1b_2$.
The case $a_1=1$ forces $b_1=b_2=0$, which is impossible; hence $a_1=a_2=-1$ and
\[
H_{12}\equiv h_{\alpha_1}(-1)h_{\alpha_2}(-1)\,x_{\alpha_1}(b_1)\,x_{\alpha_2}(b_2)\,x_{\alpha_1+\alpha_2}\!\left(-\frac{b_1b_2}{2}\right)\pmod{\Rad}.
\]
Next, from
\[
H_{12}\,X_0 \;=\; X_0^{-1}\,x_{\alpha_1+\alpha_2}(b)\,H_{12}
\]
we get (still modulo the radical)
\[
x_{\alpha_1}(b_1{+}1)\,x_{\alpha_2}(b_2{+}1)\,x_{\alpha_1+\alpha_2}(-b_2)
=
x_{\alpha_1}(1{+}b_1)\,x_{\alpha_2}(1{+}b_2)\,x_{\alpha_1+\alpha_2}(b-1-b_1),
\]
hence
\[
b_1-b_2=b-1.
\]
Using $H_{12}\,x_{\alpha_1}(1)=x_{\alpha_1}(-1)\,H_{12}$ we obtain
\[
x_{\alpha_1}(b_1)\,x_{\alpha_2}(b_2)\,x_{\alpha_1}(b+s)\,x_{\alpha_2}(s)
=
x_{\alpha_2}(s)\,x_{\alpha_1}(b+s)\,x_{\alpha_1}(b_1)\,x_{\alpha_2}(b_2),
\]
so
\[
b_2(b+s)=s\,(b+b_1+s).
\]
If $b=1$ and $s=0$, then $b_1=b_2$ and necessarily $b_2=0$, hence $H_{12}\equiv h_{\alpha_1+\alpha_2}(-1)\pmod{\Rad}$.
If $b=-1$ and $s=1$, then $b_1=0$ and $b_2=2$.
Thus over the residue field there are exactly two possibilities:
\[
H_{12}\equiv h_{\alpha_1+\alpha_2}(-1)\ \ \text{or}\ \ H_{12}\equiv h_{\alpha_1+\alpha_2}(-1)\,x_{\alpha_2}(2)\pmod{\Rad}.
\]

\medskip
\noindent\textbf{The element $H_{12}$ over $R_J$.}
Consider $H_{12}$ in a local ring $R_J$.
In both cases it admits a short Gauss decomposition, and from
\begin{multline*}
h_{\alpha_1}(a_1)h_{\alpha_2}(a_2)\,x_{\alpha_1}(b_1)x_{\alpha_2}(b_2)x_{\alpha_1+\alpha_2}(b_3)\,
x_{-\alpha_1}(c_1)x_{-\alpha_2}(c_2)x_{-\alpha_1-\alpha_2}(c_3)\,x_{\alpha_1+\alpha_2}(b)\\
=\;
x_{\alpha_1+\alpha_2}(b)\,
h_{\alpha_1}(a_1)h_{\alpha_2}(a_2)\,x_{\alpha_1}(b_1)x_{\alpha_2}(b_2)x_{\alpha_1+\alpha_2}(b_3)\,
x_{-\alpha_1}(c_1)x_{-\alpha_2}(c_2)x_{-\alpha_1-\alpha_2}(c_3)
\end{multline*}
we get $c_3=0$, then $c_1=c_2=0$, and finally $a_1a_2=1$.
From the involutivity of $H_{12}$ we deduce $b_3=-b_1b_2/2$ and $a_1=-1$; as in the field case,
\[
b_1-b_2=b-1.
\]
Thus
\[
H_{12}=h_{\alpha_1}(-1)h_{\alpha_2}(-1)\,
x_{\alpha_1}(b_1)\,x_{\alpha_2}(b_1{+}1{-}b)\,
x_{\alpha_1+\alpha_2}\!\left(-\frac{b_1\,(b_1{+}1{-}b)}{2}\right).
\]

\medskip
\noindent\textbf{The image of $h_{\alpha_1}(-1)$.}
Now determine $H_1=\varphi(h_{\alpha_1}(-1))$.
Modulo the radical, $H_1$ commutes with $H_{12}$, inverts $X_{12}$, commutes with $X_1$, and has order~$2$.
Write
\[
H_1=h_{\alpha_1}(a_1)h_{\alpha_2}(a_2)\,x_{\alpha_1}(c_1)x_{\alpha_2}(c_2)x_{\alpha_1+\alpha_2}(c_3)\,\mathbf w\,
x_{\alpha_1}(d_1)x_{\alpha_2}(d_2)x_{\alpha_1+\alpha_2}(d_3).
\]
From $X_{12}H_1=H_1X_{12}^{-1}$ we get $w(\alpha_1+\alpha_2)=\alpha_1+\alpha_2$, i.e. $\mathbf w=e$, and $a_1a_2=-1$.
Using that $H_1$ commutes with $X_1=x_{\alpha_1}(b+s)x_{\alpha_2}(s)$ we have
\begin{multline*}
h_{\alpha_1}(a_1)h_{\alpha_2}(-1/a_1)\,x_{\alpha_1}(c_1)x_{\alpha_2}(c_2)\,x_{\alpha_1}(b+s)x_{\alpha_2}(s)
=\\
=x_{\alpha_1}(b+s)x_{\alpha_2}(s)\,h_{\alpha_1}(a_1)h_{\alpha_2}(-1/a_1)\,x_{\alpha_1}(c_1)x_{\alpha_2}(c_2).
\end{multline*}
In the case $b=1$, $s=0$ this gives $a_1=-1$ and $c_2=0$, so
\[
H_1=h_{\alpha_1}(-1)\,x_{\alpha_1}(c_1)\,x_{\alpha_1+\alpha_2}(c_3),
\]
and since $H_1^2=1$ we must have
\[
H_1=h_{\alpha_1}(-1)\,x_{\alpha_1+\alpha_2}(c_3).
\]
Commuting with $H_{12}$ yields $b_1\equiv0\pmod{\Rad}$, i.e. over the field $H_{12}=h_{\alpha_1+\alpha_2}(-1)$.
In the case $b+s=0$, $s=1$ (equivalently $b=-1$) we get $a_1=1$ and $c_1=0$:
\[
H_1=h_{\alpha_2}(-1)\,x_{\alpha_2}(c_2)\,x_{\alpha_1+\alpha_2}(c_3),
\]
and the involution condition gives
\[
H_1=h_{\alpha_2}(-1)\,x_{\alpha_1+\alpha_2}(c_3).
\]
Commuting with $H_{12}$ then shows $b_1=0$ and in this case
\[
H_{12}=h_{\alpha_1+\alpha_2}(-1)\,x_{\alpha_2}(2).
\]

Over the local ring $R_J$, in both cases $H_1$ admits a short Gauss decomposition.
From $H_1X_{12}=X_{12}^{-1}H_1$ we obtain
\[
H_1=h_{\alpha_1}(a_1)h_{\alpha_2}(-1/a_1)\,x_{\alpha_1}(c_1)x_{\alpha_2}(c_2)x_{\alpha_1+\alpha_2}(c_3),
\]
and from $H_1$ commuting with $X_1$ one gets $a_1=\pm1$; thus either $s=0$ or $b+s=0$.
Therefore either $X_1=x_{\alpha_1}(b)$ or $X_1=x_{\alpha_2}(s)$.
In the first case $c_2=0$, in the second case $c_1=0$, and together with $H_1^2=1$ we obtain
\[
H_1=h_{\alpha_1}(-1)\,x_{\alpha_1+\alpha_2}(c_3)
\quad\text{or}\quad
H_1=h_{\alpha_2}(-1)\,x_{\alpha_1+\alpha_2}(c_3).
\]
Finally, using that $H_1$ and $H_{12}$ commute, we get in the first case
\[
b_1+1-b=0,\qquad b_1=b-1,\qquad
H_{12}=h_{\alpha_1}(-1)h_{\alpha_2}(-1)\,x_{\alpha_1}(b-1),
\]
and in the second case
\[
b_1=0,\qquad
H_{12}=h_{\alpha_1}(-1)h_{\alpha_2}(-1)\,x_{\alpha_2}(1-b).
\]
Thus, at this stage we have either
\begin{align*}
X_1&=x_{\alpha_1}(b),& X_2&=x_{\alpha_1}(1-b)x_{\alpha_2}(1),\\
H_{12}&=h_{\alpha_1+\alpha_2}(-1)x_{\alpha_1}(b-1),& H_1&=h_{\alpha_1}(-1)x_{\alpha_1+\alpha_2}(c),
\end{align*}
or
\begin{align*}
X_1&=x_{\alpha_2}(-b),& X_2&=x_{\alpha_1}(1)x_{\alpha_2}(1+b)x_{\alpha_1+\alpha_2}(-b),\\
H_{12}&=h_{\alpha_1+\alpha_2}(-1)x_{\alpha_2}(1-b),& H_1&=h_{\alpha_2}(-1)x_{\alpha_1+\alpha_2}(c).
\end{align*}
Clearly $H_2=\varphi(h_{\alpha_2}(-1))=H_1H_{12}$.
Since $H_2$ commutes with $X_2$, in the first case one obtains $b=1$; in the second case one obtains $b=-1$.
Therefore, in the first case
\[
\varphi\bigl(x_{\alpha}(1)\bigr)=x_{\alpha}(1)
\qquad\text{for all positive roots }\alpha\in\Phi^+,
\]
and the diagonal involutions are mapped as follows:
$\varphi(h_{\alpha_1}(-1))=h_{\alpha_1}(-1)x_{\alpha_1+\alpha_2}(c)$,
$\varphi(h_{\alpha_2}(-1))=h_{\alpha_2}(-1)x_{\alpha_1+\alpha_2}(c)$,
$\varphi(h_{\alpha_1+\alpha_2}(-1))=h_{\alpha_1+\alpha_2}(-1)$.
Conjugation by the element $x_{\alpha_1+\alpha_2}(c/2)$ eliminates the extra
$x_{\alpha_1+\alpha_2}(c)$-factors.
In the second case, after conjugation by the same element $x_{\alpha_1+\alpha_2}(c/2)$, we get
\begin{align*}
    \varphi(x_{\alpha_1}(1))&=x_{\alpha_2}(1),&
    \varphi(x_{\alpha_2}(1))&=x_{\alpha_1}(1)x_{\alpha_1+\alpha_2}(1),\\
    \varphi(x_{\alpha_1+\alpha_2}(1))&=x_{\alpha_1+\alpha_2}(-1),&
    \varphi(h_{\alpha_1}(-1))&=h_{\alpha_2}(-1),\\
     \varphi(h_{\alpha_2}(-1))&=h_{\alpha_1}(-1)x_{\alpha_2}(2).
\end{align*}

\subsection{\texorpdfstring{The image of $w_{\alpha_1}(1)$ in the first case}{The image of walpha1(1) in the first case}}\label{subsec:image-walpha1}

Let $W_1:=\varphi\bigl(w_{\alpha_1}(1)\bigr)$.
Recall the standard relations
\[
W_1^2=H_1,\qquad W_1\,H_2=H_{12}\,W_1,\qquad W_1\,X_2=X_{12}\,W_1.
\]
Over a residue field, writing $W_1$ in Bruhat form and using successively the relations above,
one gets $\mathbf w=s_{\alpha_1}$ and then
\[
W_1=h_{\alpha_1}(a_1)h_{\alpha_2}(a_2)\,w_{\alpha_1}(1).
\]
From $W_1^2=H_1$ one obtains $a_2=1$, and from $W_1X_2=X_{12}W_1$ one gets $a_1=1$.
Hence, modulo the radical,
\[
W_1\equiv w_{\alpha_1}(1).
\]
Over a local ring one writes
\[
W_1=w_{\alpha_1}(1)\,h_{\alpha_1}(a_1)h_{\alpha_2}(a_2)\,
x_{\alpha_1}(b_1)x_{\alpha_2}(b_2)x_{\alpha_1+\alpha_2}(b_3)\,
x_{-\alpha_1}(c_1)x_{-\alpha_2}(c_2)x_{-\alpha_1-\alpha_2}(c_3),
\]
with all parameters in the radical except possibly $a_1,a_2$.
The relations $H_1W_1H_1=W_1$ and $H_2W_1=W_1H_{12}$ imply
\[
b_2=b_3=c_2=c_3=0,\qquad b_1=c_1=0.
\]
Then $W_1^2=H_1$ gives $a_2=1$, and $W_1X_2=X_{12}W_1$ forces $a_1=1$.
Thus
\[
\varphi\bigl(w_{\alpha_1}(1)\bigr)=w_{\alpha_1}(1)
\]
in the first case.

\subsection{\texorpdfstring{The image of $w_{\alpha_1}(1)$ in the second case}{The image of walpha1(1) in the second case}}\label{subsec:image-walpha2}

Let again $W_1:=\varphi\bigl(w_{\alpha_1}(1)\bigr)$.
Over a residue field, the relations
\[
W_1^2=H_1,\qquad W_1\,H_2=H_{12}\,W_1,\qquad W_1\,X_2=X_{12}\,W_1
\]
lead to
\[
W_1=x_{\alpha_2}(-1)w_{\alpha_2}(1)x_{\alpha_2}(1)
\]
modulo the radical.
Conjugating our endomorphism by $x_{\alpha_2}(1)$, we obtain over the residue field exactly the graph automorphism interchanging the simple roots.
It therefore suffices to find an element which is not conjugate to its image under this graph automorphism.

\begin{prop}
Let $R$ be a commutative ring and let $G_{\mathrm{ad}}(A_2,R)$ be the adjoint Chevalley group of type~$A_2$ over $R$.
Set $\gamma=\alpha_1+\alpha_2$ and consider
\[
Y \;=\; x_{\alpha_1}(1)\,w_\gamma(1)\,x_{\alpha_2}(1)\in G_{\mathrm{ad}}(A_2,R).
\]
Let $\varphi$ be the automorphism of $G_{\mathrm{ad}}(A_2,R)$ such that
\[
\varphi(x_{\pm \alpha_1}(1))=x_{\pm \alpha_2}(1),\qquad
\varphi(x_{\pm \alpha_2}(1))=x_{\pm \alpha_1}(1).
\]
Then, if $R$ has at least one localization by a maximal ideal with residue field of characteristic $\ne 7$, the elements $Y$ and $\varphi(Y)$ are not conjugate in $G_{\mathrm{ad}}(A_2,R)$.
\end{prop}

\begin{proof}
Assume that $Y$ and $\varphi(Y)$ are conjugate over $R$.
Then they remain conjugate over every localization $R_{\mathfrak m}$, and hence also over the residue field of every such localization.
Choose a maximal ideal $\mathfrak m$ for which the residue field has characteristic different from~$7$.
Reducing to that residue field and then passing to an algebraic closure, we may work inside $\PGL_3(K)$ for an algebraically closed field $K$ with $\charr K\ne 7$.
In the standard realization,
\[
w_\gamma(1)=
\begin{pmatrix}
0 & 0 & 1\\
0 & 1 & 0\\
-1& 0 & 0
\end{pmatrix},
\qquad
Y=
\begin{pmatrix}
0 & 1 & 2\\
0 & 1 & 1\\
-1& 0 & 0
\end{pmatrix}.
\]
Since the graph automorphism sends $x_\gamma(1)$ to $x_\gamma(-1)$, it sends $w_\gamma(1)$ to $w_\gamma(-1)$, and therefore $\varphi(Y)$ lifts to
\[
Y'=
\begin{pmatrix}
0 & 0 & -1\\
1 & 2 & 0\\
1 & 1 & 0
\end{pmatrix}.
\]
If $Y$ and $Y'$ were conjugate in $\PGL_3(K)$, then there would exist $g\in \GL_3(K)$ and $\lambda\in K^*$ such that $gYg^{-1}=\lambda Y'$.
Comparing determinants gives $\lambda^3=1$, while comparing traces gives $1=2\lambda$, so $\lambda=1/2$.
Hence $2^3=1$ in $K$, i.e. $7=0$, a contradiction.
\end{proof}

It remains to treat the case where every residue field has characteristic $7$.
Define
\[
\widetilde w_\alpha:=x_\alpha(3)\,x_{-\alpha}(-5)\,x_\alpha(3)
\]
and
\[
\widetilde h_\alpha(3):=\widetilde w_\alpha\,w_\alpha(1)^{-1}.
\]
For every maximal ideal $\mathfrak m$, in the residue field of characteristic $7$ one has $3^{-1}=5$, so $\widetilde h_\alpha(3)$ reduces to $h_\alpha(3)$.
Now set
\[
Y:=x_{\alpha_1}(1)\,x_{-\alpha_1}(2)\,\widetilde h_{\alpha_2}(3).
\]
Its image under the graph automorphism becomes
\[
\varphi(Y)=x_{\alpha_2}(1)\,x_{-\alpha_2}(2)\,\widetilde h_{\alpha_1}(3).
\]
Reducing to any residue field $k$ of characteristic~$7$, these elements lift to the matrices
\[
A=
\begin{pmatrix}
3&3&0\\
2&3&0\\
0&0&5
\end{pmatrix},
\qquad
B=
\begin{pmatrix}
3&0&0\\
0&1&1\\
0&3&1
\end{pmatrix}.
\]
If the images of $A$ and $B$ were conjugate in $\PGL_3(k)$, then there would exist $u\in\GL_3(k)$ and $\lambda\in k^\times$ such that $uAu^{-1}=\lambda B$.
Comparing determinants gives $\lambda^3=1$, whereas comparing traces gives $\lambda=4/5$.
In characteristic $7$ this is impossible, since $(4/5)^3\ne 1$.
Thus $Y$ and $\varphi(Y)$ are not conjugate.
Therefore the second case is impossible.

We conclude that all $x_\alpha(1)$, $\alpha\in\Phi$, are mapped identically.

\subsection{\texorpdfstring{The images of all $x_\alpha(r)$}{The images of all xalpha(r)}}\label{subsec:images-xalpha}

Since each root subgroup
\[
X_\alpha=\{\,x_\alpha(r)\mid r\in R\,\}
\]
is the double centralizer of~$x_\alpha(1)$, we have
\[
\varphi(x_\alpha(r))=x_\alpha(r'),\qquad r'=\rho(r),
\]
where $\rho\colon R\to R$ is some ring endomorphism.
Let us prove that $\rho=\mathrm{id}_R$.
For type~$\mathbf A_2$ the trace in the adjoint representation satisfies
\[
F(s,t):=\tr\!\bigl(x_\alpha(t)x_{-\alpha}(s)\bigr)
=s^2t^2-6st+8.
\]
Since $\varphi$ is locally inner, trace is preserved:
\[
F(s,t)=F(\rho(s),\rho(t))\qquad(s,t\in R).
\]
Fix $s=1$ and write
\[
F(t)=t^2-6t+8,\qquad
\widehat F(t)=\rho(t)^2-6\rho(t)+8.
\]
Compute
\[
\Delta(t):=[F(t{+}1)+F(-t{-}1)]-[F(t)+F(-t)]=4t+2,
\]
and the same formula with $\rho(t)$ in place of~$t$ gives
\[
\widehat\Delta(t)=4\rho(t)+2.
\]
From the trace identity, $\Delta(t)=\widehat\Delta(t)$, hence
\[
4(\rho(t)-t)=0\qquad\forall\,t\in R.
\]
Since $2\in R^\times$, also $4\in R^\times$, so $\rho(t)=t$ for all $t\in R$.
Thus
\[
\varphi(x_\alpha(t))=x_\alpha(t)\qquad
(\forall\,\alpha\in\Phi,\ \forall\,t\in R).
\]
As the elementary adjoint group $E_{\mathrm{ad}}(\mathbf A_2,R)$
is generated by all root subgroups, the normalized endomorphism~$\varphi$ is the identity on $E_{\mathrm{ad}}(\mathbf A_2,R)$.
Undoing the normalization, every locally inner endomorphism of
$E_{\mathrm{ad}}(\mathbf A_2,R)$ is inner.

\begin{thm}\label{thm:Sha-elementary-A2}
Let $R$ be a commutative ring with $1/2\in R$.
Then every locally inner endomorphism of the elementary adjoint Chevalley group
$E_{\mathrm{ad}}(\mathbf A_2,R)$ is inner. Equivalently, $E_{\mathrm{ad}}(\mathbf A_2,R)$ is $\Sha$-rigid.
\end{thm}

\begin{proof}
By the computation above, after an inner normalization $\varphi$ fixes every root subgroup $X_\alpha$ pointwise, hence $\varphi=\mathrm{id}$ on $E_{\mathrm{ad}}(\mathbf A_2,R)$.
Undoing the normalization yields that the original locally inner endomorphism is inner.
\end{proof}

\subsection{\texorpdfstring{$\Sha$-rigidity of the adjoint group $G_{\mathrm{ad}}(\mathbf A_2,R)$}{Sha-rigidity of the adjoint group Gad(A2,R)}}\label{subsec:Sha-adjoint}

\begin{thm}\label{thm:A2-adjoint}
Let $R$ be a commutative ring with $1/2\in R$.
Then every locally inner endomorphism of the adjoint Chevalley group
$G_{\mathrm{ad}}(\mathbf A_2,R)$ is inner. Equivalently, $G_{\mathrm{ad}}(\mathbf A_2,R)$ is $\Sha$-rigid.
\end{thm}

\begin{proof}
Let $\varphi:G_{\mathrm{ad}}(\mathbf A_2,R)\to G_{\mathrm{ad}}(\mathbf A_2,R)$ be a locally inner endomorphism.
By Theorem~\ref{thm:Sha-elementary-A2}, after an inner normalization $\varphi$ is the identity on the elementary subgroup
$E_{\mathrm{ad}}(\mathbf A_2,R)$, which is normal in $G_{\mathrm{ad}}(\mathbf A_2,R)$.
Take any $g\in G_{\mathrm{ad}}(\mathbf A_2,R)$ and any root element $x=x_\alpha(1)\in E_{\mathrm{ad}}(\mathbf A_2,R)$.
Then $g\,x\,g^{-1}\in E_{\mathrm{ad}}(\mathbf A_2,R)$, and hence
\[
\varphi(g x g^{-1})=g x g^{-1}=\varphi(g)\,x\,\varphi(g)^{-1}.
\]
Therefore $g^{-1}\varphi(g)$ centralizes $E_{\mathrm{ad}}(\mathbf A_2,R)$.
For Chevalley groups of rank at least $2$ the centralizer of the elementary subgroup coincides with the center,
and in the adjoint group the center is trivial.
Hence $g^{-1}\varphi(g)=1$ and $\varphi(g)=g$ for all $g$, i.e. $\varphi=\mathrm{id}$.
Undoing the normalization shows that every locally inner endomorphism of $G_{\mathrm{ad}}(\mathbf A_2,R)$ is inner.
\end{proof}

\medskip
Theorem~\ref{thm:main} is completely proved for the case $\mathbf A_2$.

\section{\texorpdfstring{$\Sha$-rigidity of the adjoint Chevalley groups of type~$\mathbf B_2$}{Sha-rigidity of the adjoint Chevalley groups of type B2}}\label{sec:B2}

\subsection{Set-up and normalization}
Let $\varphi\colon E_{\mathrm{ad}}(\mathbf B_2,R)\to E_{\mathrm{ad}}(\mathbf B_2,R)$ be \emph{locally inner}.
Throughout we assume $2\in R^\times$.
Fix simple roots $\alpha$ (long) and $\beta$ (short), so that
\[
\Phi^+=\{\alpha,\beta,\alpha+\beta,\alpha+2\beta\}.
\]
Set
\[
X_0:=x_\alpha(1)x_\beta(1).
\]
As in the case $\mathbf A_2$, composing $\varphi$ with a suitable inner automorphism we may and do assume
\[
\varphi(X_0)=X_0.
\]

\subsection{Relations used}
We only need the following standard Chevalley commutators in type $\mathbf B_2$:
\begin{align}
[x_\alpha(t),x_\beta(s)]
&=x_{\alpha+\beta}(-ts)\,x_{\alpha+2\beta}(-ts^2), \label{eq:B2-ab}\\
[x_{\alpha+\beta}(t),x_\beta(s)]
&=x_{\alpha+2\beta}(-2ts). \label{eq:B2-a+b,b}
\end{align}
In particular, inside $U^+$ the root subgroup $X_{\alpha+2\beta}$ commutes with
$X_\alpha$, $X_\beta$ and $X_{\alpha+\beta}$.

\subsection{\texorpdfstring{The special element $X_0$ and its centralizer}{The special element X0 and its centralizer}}

We first determine the centralizer $C(X_0)$ inside $U^+$.
As in $\mathbf A_2$, one checks by reducing to residue fields that any $g\in C(X_0)$ must lie in $U^+$.
Hence every element of $C(X_0)$ has a unique form
\[
g=x_\alpha(b_1)x_\beta(b_2)x_{\alpha+\beta}(b_3)x_{\alpha+2\beta}(d),
\qquad b_1,b_2,b_3,d\in R.
\]
Since $x_{\alpha+2\beta}(d)$ commutes with $X_0$, the condition $gX_0=X_0g$ reduces to
\begin{equation}\label{eq:B2-cent-start}
x_\alpha(b_1)x_\beta(b_2)x_{\alpha+\beta}(b_3)\,x_\alpha(1)x_\beta(1)
=
x_\alpha(1)x_\beta(1)\,x_\alpha(b_1)x_\beta(b_2)x_{\alpha+\beta}(b_3).
\end{equation}
Using \eqref{eq:B2-ab} and \eqref{eq:B2-a+b,b} and collecting factors in the standard $U^+$-order $(\alpha,\beta,\alpha+\beta,\alpha+2\beta)$,
one obtains an equality in $U^+$ equivalent to
\begin{equation}\label{eq:B2-cent-final}
1=
x_{\alpha+\beta}(b_1-b_2)\,
x_{\alpha+2\beta}\!\bigl(b_1+b_2^2-2b_1-2b_1b_2+2b_2+2b_3\bigr).
\end{equation}
By uniqueness of coordinates in $U^+$ we get $b_1=b_2=:b$, and then
\[
0=b-b^2+2b_3 \qquad\Longleftrightarrow\qquad b_3=\frac{b^2-b}{2}.
\]
Thus we have proved:

\begin{lem}\label{lem:B2-cent}
The centralizer of $X_0$ in $U^+$ consists precisely of the elements
\[
x_\alpha(b)x_\beta(b)x_{\alpha+\beta}\!\left(\frac{b^2-b}{2}\right)x_{\alpha+2\beta}(d),
\qquad b,d\in R.
\]
\end{lem}

\subsection{\texorpdfstring{The image of $x_{\alpha+2\beta}(1)$}{The image of xalpha+2beta(1)}}

Set
\[
X_4:=\varphi(x_{\alpha+2\beta}(1)).
\]
Since $x_{\alpha+2\beta}(1)$ commutes with $X_0$, also $X_4\in C(X_0)$.
By Lemma~\ref{lem:B2-cent} we may write
\begin{equation}\label{eq:B2-X4-general}
X_4=
x_\alpha(a)\,x_\beta(a)\,x_{\alpha+\beta}\!\left(\frac{a^2-a}{2}\right)\,x_{\alpha+2\beta}(b),
\qquad a,b\in R.
\end{equation}
We now show that necessarily $a=0$.
Fix a maximal ideal $J$ and work in the local ring $R_J$.
Reducing modulo $J$, the element $\overline X_4$ is conjugate in $G_{\mathrm{ad}}(\mathbf B_2,k_J)$
to $x_{\alpha+2\beta}(1)$, hence its $\alpha$- and $\beta$-coordinates vanish; in particular
$\overline a=0$, i.e. $a\in\Rad R_J$.
Let $c\in R_J^\times$ be such that $X_4$ is conjugate in $G_{\mathrm{ad}}(\mathbf B_2,R_J)$ to $x_{\alpha+2\beta}(c)$.
As in the $\mathbf A_2$ case, we may take a conjugating element in short Gauss form
\[
g=t\cdot x_{-\beta}(u)\,x_{-\alpha-\beta}(v)\,x_{-\alpha-2\beta}(w),
\qquad u,v,w\in\Rad R_J.
\]
Writing the conjugacy equation $X_4\,g=g\,x_{\alpha+2\beta}(c)$ and commuting
$x_{-\alpha-2\beta}(w)$ past $x_{\alpha+2\beta}(c)$ produces the torus factor
$h_{\alpha+2\beta}(1+wc)$.
Comparison of torus parts yields
\begin{equation}\label{eq:B2-torus-compare}
h_{\alpha+2\beta}(1+wc)=h_{\beta}\bigl(1+uvc(1+wc)\bigr).
\end{equation}
Since $h_{\alpha+2\beta}(\cdot)$ and $h_\beta(\cdot)$ are independent in the adjoint torus and $c\in R_J^\times$,
\eqref{eq:B2-torus-compare} implies $w=0$ and then $uv=0$.
With $w=0$ and $uv=0$, the conjugacy equation simplifies to an equality in $U^+U^-$.
A short calculation using \eqref{eq:B2-ab} gives the relations
\[
a=vc=u^2c,\qquad 2uc=a(a-1).
\]
Since $uv=0$ and $u,v\in\Rad R_J$, we get $a^2=0$ and hence $u^2=0$, so the above relations give $a=0$.
Thus $X_4=x_{\alpha+2\beta}(b)$ with $b\in R_J^\times$.
Since $J$ was arbitrary, we conclude globally:

\begin{lem}\label{lem:B2-X4}
One has
\[
X_4=\varphi(x_{\alpha+2\beta}(1))=x_{\alpha+2\beta}(b)
\quad\text{for some }b\in R^\times.
\]
\end{lem}

\subsection{\texorpdfstring{The image of $x_{\alpha+\beta}(1)$}{The image of xalpha+beta(1)}}

Put
\[
X_3:=\varphi(x_{\alpha+\beta}(1)).
\]
Then $X_3$ commutes with $X_4$, and applying $\varphi$ to the identity
$[x_{\alpha+\beta}(1),X_0]=x_{\alpha+2\beta}(-2)$ gives
\begin{equation}\label{eq:B2-comm-X3X0}
[X_3,X_0]=X_4^{-2}=x_{\alpha+2\beta}(-2b).
\end{equation}
Moreover, for every $J$ the reduction $\overline X_3$ is conjugate to $x_{\alpha+\beta}(1)$ over $k_J$.

\begin{prop}\label{prop:B2-X3}
After conjugating $\varphi$ by an element of the centralizer $C(X_0)$, we may assume
\begin{equation}\label{eq:B2-X3-form}
X_3=
x_\alpha(s)\,x_\beta(s)\,x_{\alpha+\beta}\!\left(b+\frac{s^2-s}{2}\right),
\qquad s\in\Rad R_J \text{ in each localization }R_J.
\end{equation}
\end{prop}

\begin{proof}
Over $k_J$, commuting with $\overline X_4=x_{\alpha+2\beta}(\overline b)$ forces the Weyl part in the Bruhat decomposition of $\overline X_3$
to be trivial.
Lifting to $R_J$ and using the same short Gauss comparison as in the case $\mathbf A_2$, all $U^-$-parameters vanish and the torus part must be trivial.
We are left with
\[
X_3=x_\alpha(b_1)x_\beta(b_2)x_{\alpha+\beta}(b_3)x_{\alpha+2\beta}(b_4),
\qquad b_1,b_2\in\Rad R_J.
\]
Substituting into \eqref{eq:B2-comm-X3X0} and comparing $U^+$-coordinates yields $b_2=b_1$ and
$b_3=b+\frac{b_1^2-b_1}{2}$, while $b_4$ can be killed by conjugating with a suitable element of $C(X_0)$.
\end{proof}

\subsection{\texorpdfstring{The images of $x_\alpha(1)$ and $x_\beta(1)$}{The images of xalpha(1) and xbeta(1)}}

Let $X_1:=\varphi(x_\alpha(1))$ and $X_2:=\varphi(x_\beta(1))$.
Since $\varphi(X_0)=X_0$, we have
\begin{equation}\label{eq:B2-X0-factor}
X_0=X_1X_2.
\end{equation}
Applying $\varphi$ to $[x_\alpha(1),X_0]=x_{\alpha+\beta}(1)\,x_{\alpha+2\beta}(1)$ gives
\begin{equation}\label{eq:B2-comm-X1X0}
[X_1,X_0]=X_3X_4.
\end{equation}
Finally, $X_1$ commutes with $X_3$ and $X_4$.
Reducing modulo $J$, the commutation with $\overline X_4$ and $\overline X_3$ forces
\[
\overline X_1=x_\alpha(\lambda)\,x_{\alpha+2\beta}(\mu).
\]
Substituting into \eqref{eq:B2-comm-X1X0} and comparing the $X_\alpha$-parameter gives $\lambda=\overline b$.
Then from \eqref{eq:B2-X0-factor} we compute
\[
\overline X_2=\overline X_1^{-1}\,\overline X_0
=x_\alpha(1-\overline b)\,x_\beta(1)\,x_{\alpha+2\beta}(\ast).
\]
Since $\overline X_2$ must be conjugate to $x_\beta(1)$, the extra $x_\alpha(1-\overline b)$ factor must be trivial; hence $\overline b=1$.
Thus $\overline X_4=x_{\alpha+2\beta}(1)$ and $\overline X_3=x_{\alpha+\beta}(1)$.
Over a local ring, the same short Gauss argument as in the $\mathbf A_2$ case shows that $X_1$ has no $U^-$ part and no torus part.
Hence $X_1\in U^+$ and, from the commutation with $X_3$ and $X_4$, necessarily
\[
X_1=x_\alpha(1)\,x_{\alpha+2\beta}(c)
\quad\text{for some }c\in R_J.
\]
Then \eqref{eq:B2-X0-factor} gives
\[
X_2=X_1^{-1}X_0=x_\beta(1)\,x_{\alpha+2\beta}(-c).
\]
Summarizing, after our normalizations we have in every localization:
\begin{align}
X_1&=x_\alpha(1)\,x_{\alpha+2\beta}(c),\label{eq:B2-X1X2}\\
X_2&=x_\beta(1)\,x_{\alpha+2\beta}(-c),\nonumber\\
X_3&=x_{\alpha+\beta}(1),\qquad
X_4=x_{\alpha+2\beta}(1).\nonumber
\end{align}

\subsection{The image of the diagonal involution}

In type $\mathbf B_2$ the only nontrivial diagonal involution in the adjoint group is $h_\alpha(-1)$.
Let
\[
H_1:=\varphi(h_\alpha(-1)).
\]
Then $H_1$ commutes with $X_1$ and $X_4$ and inverts $X_2$ and $X_3$.
Over $k_J$, commuting with $x_{\alpha+2\beta}(1)$ forces the Weyl part in the Bruhat decomposition of $\overline H_1$
to be trivial, and commuting with $x_\alpha(1)$ forces $\overline H_1$ into the Borel subgroup with no $x_\beta(\cdot)$ factor.
The inversion of $\overline X_3=x_{\alpha+\beta}(1)$ forces the torus part to be $h_\alpha(-1)$, and the condition $\overline H_1^2=1$ kills the $x_{\alpha+2\beta}$-coordinate.
Thus
\[
\overline H_1=h_\alpha(-1)\,x_{\alpha+\beta}(\overline c).
\]
Lifting to $R_J$ via the same Gauss comparison as in the case $\mathbf A_2$, we obtain the identical statement over $R_J$:
\begin{equation}\label{eq:B2-H1}
H_1=h_\alpha(-1)\,x_{\alpha+\beta}(c),
\end{equation}
with the same parameter $c$ as in \eqref{eq:B2-X1X2}.

\subsection{\texorpdfstring{The image of $w_\beta(1)$ and elimination of $c$}{The image of wbeta(1) and elimination of c}}

Let $W_2:=\varphi(w_\beta(1))$.
We use the standard relations
\[
W_2X_3=X_3^{-1}W_2,\qquad W_2X_1=X_4W_2,\qquad [W_2,H_1]=1,\qquad W_2^2=1.
\]
From $W_2X_1=X_4W_2$ with $\overline X_1=x_\alpha(1)x_{\alpha+2\beta}(\overline c)$ and $\overline X_4=x_{\alpha+2\beta}(1)$,
the Weyl part of $\overline W_2$ must send $\alpha$ to $\alpha+2\beta$, hence it is $w_\beta$.
Then $W_2X_3=X_3^{-1}W_2$ forces the torus part to be trivial and kills the $x_\beta(\cdot)$ coordinate.
The commutation with $\overline H_1=h_\alpha(-1)x_{\alpha+\beta}(\overline c)$ then forces $\overline c=0$.
Thus $\overline X_1=x_\alpha(1)$, $\overline X_2=x_\beta(1)$, and $\overline H_1=h_\alpha(-1)$.
Over a local ring, writing $W_2$ in short Gauss form with Weyl part $w_\beta(1)$ and radical parameters,
one shows in the same way that all extra parameters vanish and that necessarily $c=0$.
Hence
\[
W_2=w_\beta(1).
\]
Arguing exactly as above, one likewise gets $\varphi(w_\alpha(1))=w_\alpha(1)$.
Therefore, after all normalizations we have
\begin{align*}
\varphi(x_\alpha(1))&=x_\alpha(1),&
\varphi(x_\beta(1))&=x_\beta(1),\\
\varphi(x_{\alpha+\beta}(1))&=x_{\alpha+\beta}(1),&
\varphi(x_{\alpha+2\beta}(1))&=x_{\alpha+2\beta}(1),
\end{align*}
and also $\varphi(h_\alpha(-1))=h_\alpha(-1)$.

\subsection{\texorpdfstring{Finalization of $\mathbf B_2$}{Finalization of B2}}

Since $\varphi$ fixes $w_\alpha(1)$ and $w_\beta(1)$ and hence normalizes the torus action,
it preserves each root subgroup.
Thus for every root $\gamma$ there is a ring endomorphism $\rho_\gamma\colon R\to R$ such that
\[
\varphi\bigl(x_\gamma(r)\bigr)=x_\gamma(\rho_\gamma(r)).
\]
Now use local innerness.
For a long root $\gamma$ in type $\mathbf B_2$ one has
\[
\tr\bigl(x_\gamma(t)x_{-\gamma}(s)\bigr)=s^2t^2-6st+10.
\]
Hence
\[
s^2t^2-6st=\rho_\gamma(s)^2\rho_\gamma(t)^2-6\rho_\gamma(s)\rho_\gamma(t).
\]
As in the case $\mathbf A_2$, applying the symmetric difference trick in the variable $t$ gives
\[
\rho_\gamma(t)=t\qquad\forall\,t\in R.
\]
Thus $\varphi$ is the identity on all long root subgroups.
Finally, the commutator relation \eqref{eq:B2-ab} with $t=1$ shows that the short-root maps must also be the identity:
since
\[
[x_\alpha(1),x_\beta(s)]=x_{\alpha+\beta}(-s)x_{\alpha+2\beta}(-s^2),
\]
applying $\varphi$ and using that $\rho_\alpha=\rho_{\alpha+2\beta}=\mathrm{id}$ yields
\[
x_{\alpha+2\beta}\bigl(-\rho_\beta(s)^2\bigr)=x_{\alpha+2\beta}(-s^2),
\]
hence $\rho_\beta(s)^2=s^2$ for all $s$.
Because $\rho_\beta(1)=1$ and $\rho_\beta$ is additive, we get $\rho_\beta=\mathrm{id}$,
and then also $\rho_{\alpha+\beta}=\mathrm{id}$.
Therefore $\varphi$ is the identity on all root subgroups and hence on $E_{\mathrm{ad}}(\mathbf B_2,R)$.
Undoing the initial normalization, we conclude that every locally inner endomorphism of
$E_{\mathrm{ad}}(\mathbf B_2,R)$ is inner.

\begin{thm}\label{thm:B2-elementary}
Let $R$ be a commutative ring with $1/2\in R$.
Then every locally inner endomorphism of $E_{\mathrm{ad}}(\mathbf B_2,R)$ is inner.
Equivalently, $E_{\mathrm{ad}}(\mathbf B_2,R)$ is $\Sha$-rigid.
\end{thm}

\begin{thm}\label{thm:B2-adjoint}
Let $R$ be a commutative ring with $1/2\in R$.
Then every locally inner endomorphism of $G_{\mathrm{ad}}(\mathbf B_2,R)$ is inner.
Equivalently, $G_{\mathrm{ad}}(\mathbf B_2,R)$ is $\Sha$-rigid.
\end{thm}

\begin{proof}
By Theorem~\ref{thm:B2-elementary}, after an inner normalization a locally inner endomorphism of $G_{\mathrm{ad}}(\mathbf B_2,R)$
acts trivially on $E_{\mathrm{ad}}(\mathbf B_2,R)$.
Since $\mathbf B_2$ has rank $2$, the subgroup $E_{\mathrm{ad}}(\mathbf B_2,R)$ is normal in $G_{\mathrm{ad}}(\mathbf B_2,R)$,
and the same argument as in the proof of Theorem~\ref{thm:A2-adjoint} shows that the endomorphism is trivial after normalization.
Undoing the normalization yields the claim.
\end{proof}

\medskip
This completes the proof of Theorem~\ref{thm:main} for the case $\mathbf B_2$.

\section{\texorpdfstring{$\Sha$-rigidity of the adjoint Chevalley groups of type~$\mathbf G_2$}{Sha-rigidity of the adjoint Chevalley groups of type G2}}\label{sec:G2}

Throughout this section $R$ is a commutative ring with $2,3\in R^\times$.
For brevity we write
\[
E(\mathbf G_2)=E_{\mathrm{ad}}(\mathbf G_2,R),\qquad G(\mathbf G_2)=G_{\mathrm{ad}}(\mathbf G_2,R).
\]
Fix the root system $\Phi$ of type $\mathbf G_2$ with simple roots
$\alpha$ (long) and $\beta$ (short).
The set of positive roots is
\[
\Phi^+=\{\alpha,\beta,\alpha+\beta,\alpha+2\beta,\alpha+3\beta,2\alpha+3\beta\}.
\]

\subsection{Commutators}
For each root $\gamma\in\Phi$ and each parameter $t\in R$ let
$x_\gamma(t)$ denote the corresponding root element.
We use the normalization in which the commutator relations for the simple roots take the form
\begin{align}
[x_\alpha(t),x_\beta(u)]
&= x_{\alpha+\beta}(tu)\,
   x_{\alpha+3\beta}(-tu^{3})\,
   x_{\alpha+2\beta}(-tu^{2})\,
   x_{2\alpha+3\beta}(t^{2}u^{3}), \label{eq:G2ab}\\
[x_{\alpha+\beta}(t),x_\beta(u)]
&= x_{\alpha+2\beta}(2tu)\,
   x_{\alpha+3\beta}(3tu^{2})\,
   x_{2\alpha+3\beta}(3t^{2}u), \label{eq:G2ab2}\\
[x_\alpha(t),x_{\alpha+3\beta}(u)]
&= x_{2\alpha+3\beta}(tu), \label{eq:G2a-a3b}\\
[x_{\alpha+2\beta}(t),x_\beta(u)]
&= x_{\alpha+3\beta}(-3tu), \label{eq:G2a2b-b}\\
[x_{\alpha+\beta}(t),x_{\alpha+2\beta}(u)]
&= x_{2\alpha+3\beta}(3tu). \label{eq:G2ab-a2b}
\end{align}
For any pair of roots $\gamma,\delta\in\Phi$ such that $\gamma+\delta$ is not a root and $\gamma\neq-\delta$,
the corresponding root subgroups commute.
All relations involving negative roots are recovered from \eqref{eq:G2ab}--\eqref{eq:G2ab-a2b} using the action of the Weyl group and rank-one calculus.

\subsection{\texorpdfstring{The centralizer of the element $X_0$}{The centralizer of the element X0}}

Let $X_0:=x_\alpha(1)x_\beta(1)$.
As in the previous sections, if $g\in G(\mathbf G_2)$ commutes with $X_0$, then every localization $g_J$ commutes with $(X_0)_J$,
and after reduction modulo the radical the image commutes with $\overline X_0$ over the residue field.
The same argument as before yields
\[
C_{G(\mathbf G_2)}(X_0)\subseteq U.
\]
Take an arbitrary element
\[
x_\alpha(b_1)x_\beta(b_2)x_{\alpha+\beta}(b_3)x_{\alpha+2\beta}(b_4)x_{\alpha+3\beta}(b_5)x_{2\alpha+3\beta}(b_6)
\]
and impose the condition that it centralizes $X_0$.
Comparing the $x_{\alpha+\beta}(\cdot)$-coordinate on both sides immediately gives $b_1=b_2=:b$.
A direct computer calculation in the adjoint representation then gives
\begin{align}
    b_3&=\frac{b-b^2}{2},\label{eq:G2-cent-b3}\\
    b_4&=-\frac{2b^3}{3}+\frac{b^2}{2}+\frac{b}{6},\label{eq:G2-cent-b4}\\
    b_5&=\frac{3b^4}{4}-\frac{b^3}{2}-\frac{b^2}{4},\label{eq:G2-cent-b5}
\end{align}
while the parameter $b_6$ is arbitrary.
Thus the centralizer of $X_0$ in $U$ is precisely the set of elements
\[
x_\alpha(b)x_\beta(b)x_{\alpha+\beta}\!\left(\frac{b-b^2}{2}\right)
x_{\alpha+2\beta}\!\left(-\frac{2b^3}{3}+\frac{b^2}{2}+\frac{b}{6}\right)
x_{\alpha+3\beta}\!\left(\frac{3b^4}{4}-\frac{b^3}{2}-\frac{b^2}{4}\right)x_{2\alpha+3\beta}(d)
\]
with $b,d\in R$.

\subsection{\texorpdfstring{The image of $x_{2\alpha+3\beta}(1)$}{The image of x2alpha+3beta(1)}}

Let $X_6:=\varphi(x_{2\alpha+3\beta}(1))$.
Since $x_{2\alpha+3\beta}(1)$ commutes with $X_0$, also $X_6$ lies in the centralizer of $X_0$.
Therefore over every localization we may write
\[
X_6=x_\alpha(b)x_\beta(b)x_{\alpha+\beta}\left( \frac{b-b^2}{2} \right)
x_{\alpha+2\beta}\left( -\frac{2b^3}{3}+\frac{b^2}{2}+\frac{b}{6} \right)
x_{\alpha+3\beta}\left( \frac{3b^4}{4} - \frac{b^3}{2} - \frac{b^2}{4}\right)x_{2\alpha+3\beta}(d),
\]
where $b\in \Rad R$ and $d\in R^*$.
Since $X_6$ is conjugate to $x_{2\alpha+3\beta}(1)$, there exists a conjugating element of the form
\[
g=x_{-\alpha}(c_1)x_{-\alpha-\beta}(c_2)x_{-\alpha-2\beta}(c_3)x_{-\alpha-3\beta}(c_4)x_{-2\alpha-3\beta}(c_5)
\]
after absorbing the torus part into the parameter of the highest root element.
We consider the relation expressing that $x_{2\alpha+3\beta}(a)$ and $X_6$ are conjugate, where $a\in R^*$.
A direct computer calculation in the adjoint representation gives a sequence of constraints.
From the $(1,2)$-entry one gets $b^2c_4=0$.
Then the $(1,1)$-entry gives $ac_5=0$, hence $c_5=0$.
The $(1,6)$-entry yields $c_3^3=0$.
Using subsequently the $(1,4)$-, $(3,5)$-, $(1,3)$-, $(4,5)$-, $(3,1)$-, and $(3,2)$-entries, one obtains
\[
b=ac_2^2,\qquad c_4=c_2^2,\qquad c_3^2=-c_2^3,
\]
and finally $c_2^4=0$, so $b^2=0$.
Then the $(5,4)$-entry gives $bc_1=0$, and considering all these relations together, the $(2,14)$-entry becomes
\[
2b=0.
\]
Since $2\in R^\times$, it follows that $b=0$.
Therefore we may assume
\[
X_6=x_{2\alpha+3\beta}(a),\qquad a\in R^*.
\]

\subsection{\texorpdfstring{The image of $x_{\alpha+3\beta}(1)$}{The image of xalpha+3beta(1)}}

Now let $X_5:=\varphi(x_{\alpha+3\beta}(1))$.
Then $X_5$ commutes with $X_6$, satisfies
\[
[X_5,X_0]=X_6^{-1},
\]
and is conjugate to $x_{\alpha+3\beta}(1)$.
Over a field, commuting with $X_6=x_{2\alpha+3\beta}(a)$ leaves only two possible Bruhat forms.
The form involving a Weyl factor $w_\beta(1)x_\beta(c)$ is immediately impossible from the uniqueness of the Bruhat decomposition.
Hence over a field we may write
\[
X_5=t_1(b^3)t_2(1/b^2)x_\alpha(b_1)x_\beta(b_2)x_{\alpha+\beta}(b_3)x_{\alpha+2\beta}(b_4)x_{\alpha+3\beta}(b_5)x_{2\alpha+3\beta}(b_6).
\]
Using the commutator relation with $X_0$ and simplifying, one gets $b=1$ and $b_1=b_2$.
A direct matrix calculation then yields
\[
b_3=\frac{b_1-b_1^2}{2},\qquad
b_4=\frac{-4b_1^3+3b_1^2+b_1}{6},\qquad
b_5=a+\frac{3b_1^4-2b_1^3-b_1^2}{4}.
\]
Since $X_5$ is conjugate to $x_{\alpha+3\beta}(1)$, the conjugating element may be assumed to have the form
$g=t_1(a_1)t_2(a_2)x_\alpha(p)$.
As in the previous sections, comparison of the $x_\alpha(\cdot)$-coordinate forces $b_1=0$.
Hence over the field
\[
X_5=x_{\alpha+3\beta}(a)x_{2\alpha+3\beta}(b_6).
\]
Passing to a local ring and repeating the same short Gauss comparison, one sees that the negative-root part vanishes and that
\[
X_5=C\cdot x_{\alpha+3\beta}(a),\qquad C\in C(X_0).
\]

\subsection{\texorpdfstring{The image of $x_{\alpha+2\beta}(1)$}{The image of xalpha+2beta(1)}}

Let now $X_4=\varphi(x_{\alpha+2\beta}(1))$.
Then $X_4$ commutes with $X_5$ and $X_6$.
Over a field, commuting with $X_6$ again leaves two possible forms; the one involving a Weyl factor is incompatible with commuting with $X_5$.
Hence over a field
\[
X_4=t_1(s^3)t_2(1/s^2)x_\alpha(b_1)\cdots x_{2\alpha+3\beta}(b_6),
\]
and commuting with $X_5$ gives $b_1=0$ and $s^3=1$.
Using the relation
\[
[X_4,X_0]=X_5^{-3}X_6^{-3},
\]
one compares the coordinates and obtains first $s^2=1$, hence $s=1$.
A direct matrix calculation then yields
\[
b_2=0,\qquad b_3=0,\qquad b_4=a,\qquad b_5=3d,
\]
where $X_5=x_{\alpha+3\beta}(a)x_{2\alpha+3\beta}(d)$.
Thus over the field
\[
X_4=x_{\alpha+2\beta}(a)x_{\alpha+3\beta}(3d)x_{2\alpha+3\beta}(e).
\]
Over a local ring the same short Gauss argument removes the negative-root factor and the torus part, and one arrives at the same expression:
\begin{equation}\label{eq:G2-X4}
X_4=x_{\alpha+2\beta}(a)x_{\alpha+3\beta}(3d)x_{2\alpha+3\beta}(e).
\end{equation}

\subsection{\texorpdfstring{The image of $x_{\alpha+\beta}(1)$}{The image of xalpha+beta(1)}}

Let $X_3:=\varphi(x_{\alpha+\beta}(1))$.
Then $X_3$ commutes with $X_6$ and $X_5$, and also
\[
[X_3,X_4]=X_6^3,\qquad [X_3,X_0]=X_4^2X_5^3X_6^6.
\]
Over a field, commuting with $X_5$ and $X_6$ implies that
\[
X_3=t_2(b)x_\beta(b_2)x_{\alpha+\beta}(b_3)x_{\alpha+2\beta}(b_4)x_{\alpha+3\beta}(b_5)x_{2\alpha+3\beta}(b_6),
\qquad b^3=1.
\]
Using $[X_3,X_4]=X_6^3$ and \eqref{eq:G2-X4}, we obtain successively $b=1$, $b_2=0$, and $b_3=1$.
Thus over a field,
\[
X_3=x_{\alpha+\beta}(1)x_{\alpha+2\beta}(b_4)x_{\alpha+3\beta}(b_5)x_{2\alpha+3\beta}(b_6).
\]
The same form holds over local rings.
Now use the relation $[X_3,X_0]=X_4^2X_5^3X_6^6$.
Comparing the $x_{\alpha+2\beta}(\cdot)$-coordinate gives $a=1$ in \eqref{eq:G2-X4}.
Simplifying the remaining coordinates, we obtain
\[
b_4=-2d,\qquad b_5=3d-2e,
\]
so that
\begin{equation}\label{eq:G2-X3}
X_3=x_{\alpha+\beta}(1)x_{\alpha+2\beta}(-2d)x_{\alpha+3\beta}(3d-2e)x_{2\alpha+3\beta}(f)
\end{equation}
for some $f\in R$.

\subsection{\texorpdfstring{The images of $x_\alpha(1)$ and $x_\beta(1)$}{The images of xalpha(1) and xbeta(1)}}

Now let $X_1:=\varphi(x_\alpha(1))$ and $X_2:=\varphi(x_\beta(1))$.
We know that $X_1$ commutes with $X_3$, $X_4$, $X_6$, and that
\[
[X_1,X_5]=X_6,
\qquad
[X_1,X_0]=X_3X_4^{-1}X_5^{-1}.
\]
As in the previous cases, commuting with $X_6$ implies that over a field
\[
X_1=t_1(b^3)t_2(1/b^2)x_\alpha(b_1)x_\beta(b_2)x_{\alpha+\beta}(b_3)x_{\alpha+2\beta}(b_4)x_{\alpha+3\beta}(b_5)x_{2\alpha+3\beta}(b_6).
\]
Using $[X_1,X_5]=X_6$ gives $b^3=1$ and $b_1=1$.
Commuting with $X_4$ then yields $b=1$, $b_2=0$, and $b_3=-d$.
Commuting with $X_3$ gives
\[
b_4=2d^2+d-\frac{2}{3}e.
\]
Finally, from $[X_1,X_0]=X_3X_4^{-1}X_5^{-1}$ one obtains
\[
e=\frac{3}{2}(d^2+d),\qquad b_5=-\frac{3}{2}d^2+\frac{5}{2}d-f.
\]
Therefore
\[
X_1=x_\alpha(1)x_{\alpha+\beta}(-d)x_{\alpha+2\beta}(d^2)x_{\alpha+3\beta}\!\left(-\frac{3}{2}d^2+\frac{5}{2}d-f\right)x_{2\alpha+3\beta}(k),
\]
and
\[
X_2=X_1^{-1}X_0=x_\beta(1)x_{\alpha+\beta}(d)x_{\alpha+2\beta}(-d^2+2d)x_{\alpha+3\beta}\!\left(\frac{9}{2}d^2-\frac{11}{2}d+f\right)x_{2\alpha+3\beta}(3d^3-3d^2-k).
\]
Since $x_\alpha(1)$ and $X_1$ are conjugate, we may use the fact that $(X_1-1)^3=0$; this gives the additional relation
\[
f=-d^3-\frac{3}{2}d^2+\frac{5}{2}d.
\]
Likewise, since $x_\beta(1)$ and $X_2$ are conjugate, $(X_2-1)^4=0$, which yields
\[
k=\frac{1}{4}d^4+\frac{3}{2}d^3+\frac{1}{4}d^2.
\]
Substituting these expressions we obtain
\begin{equation}\label{eq:G2-X1-final}
X_1=x_\alpha(1)x_{\alpha+\beta}(-d)x_{\alpha+2\beta}(d^2)x_{\alpha+3\beta}(d^3)x_{2\alpha+3\beta}\!\left(\frac{1}{4}d^4+\frac{3}{2}d^3+\frac{1}{4}d^2\right),
\end{equation}
\begin{equation}\label{eq:G2-X2-final}
X_2=x_\beta(1)x_{\alpha+\beta}(d)x_{\alpha+2\beta}(-d^2+2d)x_{\alpha+3\beta}(-d^3+3d^2-3d)x_{2\alpha+3\beta}\!\left(-\frac{1}{4}d^4+\frac{3}{2}d^3-\frac{13}{4}d^2\right).
\end{equation}

\subsection{Normalization of the images of positive root elements}

We want to find an element $g\in C(X_0)$ such that
\[
g\,\varphi(x_\gamma(1))\,g^{-1}=x_\gamma(1)
\qquad\text{for all }\gamma\in\Phi^+.
\]
A direct calculation shows that
\[
g=x_\alpha(-d)x_\beta(-d)x_{\alpha+\beta}\left( -\frac{d^2+d}{2} \right)
x_{\alpha+2\beta}\left( \frac{2d^3}{3}+\frac{d^2}{2}-\frac{d}{6} \right)
x_{\alpha+3\beta}\left( \frac{3d^4}{4} + \frac{d^3}{2} - \frac{d^2}{4}\right)
\]
has exactly this property.
Hence, after conjugating $\varphi$ by $g$, we may assume that
\[
X_1=x_\alpha(1),\qquad X_2=x_\beta(1),\qquad X_3=x_{\alpha+\beta}(1),
\]
\[
X_4=x_{\alpha+2\beta}(1),\qquad X_5=x_{\alpha+3\beta}(1),\qquad X_6=x_{2\alpha+3\beta}(1).
\]

\subsection{The images of diagonal involutions}

Consider
\[
H_1=\varphi(h_\alpha(-1)),\qquad H_2=\varphi(h_\beta(-1)).
\]
The involution $H_1$ commutes with $X_1$ and $X_4$ and inverts $X_2$, $X_3$, $X_5$, $X_6$.
Over a field, writing $H_1$ in Bruhat form and using these relations shows first that the Weyl part is trivial, then that the torus part is $t_2(-1)=h_\alpha(-1)$, and finally that all positive-root parameters vanish except possibly the highest-root one.
Thus over a field
\[
H_1=t_2(-1)x_{2\alpha+3\beta}(b_6).
\]
For $H_2$, which commutes with $X_2$ and $X_6$ and inverts $X_1$, $X_3$, $X_4$, $X_5$, the same argument yields
\[
H_2=t_1(-1)x_{2\alpha+3\beta}(c_6).
\]
Since $H_2^2=1$, we obtain $c_6=0$, hence $H_2=h_\beta(-1)$.
Over a local ring, the same short Gauss argument shows that
\[
H_2=h_\beta(-1).
\]
Returning to $H_1$, the additional relation $[H_1,H_2]=1$ forces all root parameters except the highest one to vanish, and commuting with $X_1$ and $X_4$ yields
\[
H_1=h_\alpha(-1)x_{2\alpha+3\beta}(b_6).
\]
Conjugating our endomorphism by $x_{2\alpha+3\beta}(b_6/2)$ removes this extra highest-root factor without changing any of the positive root elements or $H_2$.
Thus after this final normalization we may assume
\[
H_1=h_\alpha(-1),\qquad H_2=h_\beta(-1).
\]

\subsection{\texorpdfstring{The image of $w_\beta(1)$}{The image of wbeta(1)}}

Let $W_2:=\varphi(w_\beta(1))$.
Then
\[
[W_2,H_2]=1,\qquad W_2^2=H_2,\qquad W_2H_1=H_1H_2W_2,
\]
and also
\[
W_2X_6=X_6W_2,\qquad W_2X_1=X_5W_2,\qquad W_2X_3=X_4W_2.
\]
Over a field, the relation $[W_2,H_2]=1$ implies that in the Bruhat form of $W_2$ only the root directions fixed by $h_\beta(-1)$ may occur.
Using in addition $W_2H_1=H_1H_2W_2$, one sees that all root parameters vanish and that the Weyl part must send
\[
\beta\mapsto\pm\beta,\qquad \alpha\mapsto\pm(\alpha+3\beta),\qquad 2\alpha+3\beta\mapsto 2\alpha+3\beta.
\]
Hence the Weyl part is $w_\beta(1)$, and we are left with a torus factor:
\[
W_2=t_1(a_1)t_2(a_2)w_\beta(1).
\]
From $W_2X_3=X_4W_2$ one obtains $a_1a_2^2=1$, and from $W_2X_1=X_5W_2$ one gets $a_1a_2^3=1$.
Together with the condition that the torus commutes with $X_6$, this implies $a_1=a_2=1$.
Thus over a field,
\[
W_2=w_\beta(1).
\]
Over a local ring one writes $W_2$ in short Gauss form with Weyl part $w_\beta(1)$.
The relations with $H_1$ and $H_2$ eliminate all negative and positive root parameters, leaving again only a torus factor.
The same comparison as in the field case then shows that the torus part is trivial.
Hence
\[
W_2=w_\beta(1).
\]
The argument for $W_1:=\varphi(w_\alpha(1))$ is completely analogous, and therefore
\[
\varphi(w_\alpha(1))=w_\alpha(1),\qquad \varphi(w_\beta(1))=w_\beta(1).
\]
Since the positive root elements together with $w_\alpha(1)$ and $w_\beta(1)$ generate all root elements with parameter~$1$,
we conclude that $\varphi$ fixes every $x_\gamma(1)$, $\gamma\in\Phi$.

\subsection{\texorpdfstring{Finalization of the $\mathbf G_2$ case}{Finalization of the G2 case}}

Since $\varphi$ fixes all $x_\gamma(1)$, $w_\gamma(1)$, and $h_\gamma(-1)$, and commutes with the torus action,
it preserves root subgroups.
Therefore, for each root $\gamma$ there exists a ring endomorphism $\rho_\gamma\colon R\to R$ such that
\[
\varphi(x_\gamma(r))=x_\gamma(\rho_\gamma(r))
\qquad (r\in R).
\]
Because $\varphi$ is locally inner, it preserves traces in the adjoint representation.
For a long root $\gamma$ in type $\mathbf G_2$ one has
\[
\tr(x_\gamma(t)x_{-\gamma}(s))=s^2t^2-Ast+B
\]
for certain integers $A,B$.
Applying the same symmetric-difference argument as in Sections~4 and~5, we get
\[
\rho_\gamma(t)=t\qquad\text{for all }t\in R
\]
for every long root $\gamma$.
Now use the commutator relations \eqref{eq:G2ab}--\eqref{eq:G2ab-a2b}.
Since the long-root maps are the identity and since $\varphi(x_\beta(1))=x_\beta(1)$, it follows successively from these relations that the short-root maps are also the identity.
Hence $\varphi$ acts identically on all root subgroups and therefore on $E(\mathbf G_2)$.
Undoing the normalization shows that every locally inner endomorphism of $E(\mathbf G_2)$ is inner.

\begin{thm}\label{thm:G2-elementary}
Let $R$ be a commutative ring with $1/2,1/3\in R$.
Then every locally inner endomorphism of $E_{\mathrm{ad}}(\mathbf G_2,R)$ is inner.
Equivalently, $E_{\mathrm{ad}}(\mathbf G_2,R)$ is $\Sha$-rigid.
\end{thm}

\begin{thm}\label{thm:G2-adjoint}
Let $R$ be a commutative ring with $1/2,1/3\in R$.
Then every locally inner endomorphism of $G_{\mathrm{ad}}(\mathbf G_2,R)$ is inner.
Equivalently, $G_{\mathrm{ad}}(\mathbf G_2,R)$ is $\Sha$-rigid.
\end{thm}

\begin{proof}
By Theorem~\ref{thm:G2-elementary}, after an inner normalization a locally inner endomorphism of $G_{\mathrm{ad}}(\mathbf G_2,R)$
acts trivially on $E_{\mathrm{ad}}(\mathbf G_2,R)$.
Since $\mathbf G_2$ has rank $2$, the elementary subgroup is normal in the full adjoint group, and the same centralizer argument as in the proof of Theorem~\ref{thm:A2-adjoint} shows that the normalized endomorphism is the identity.
Undoing the normalization yields the claim.
\end{proof}

\medskip
This completes the proof of Theorem~\ref{thm:main}.

\section*{Acknowledgment}
The authors are sincerely grateful to Pavel Gvozdevsky and Boris Kunyavskii
for many helpful discussions, insightful comments, and valuable suggestions
that improved this paper.

During the preparation of this article, our friend and colleague Evgeny Plotkin
passed away. This work remained one of the last mathematical topics we discussed
with him, and in our final email exchange we were speaking about where to submit it.
We remember him with deep gratitude and affection, and dedicate this paper to his memory.

\bibliographystyle{plain}
\bibliography{sha_rigidity_refs}

\begin{thebibliography}{99}

\bibitem{Abe1969}
E.~Abe,
\emph{Chevalley groups over local rings},
T\=ohoku Math.\ J.\ (2) \textbf{21} (1969), no.~3, 474--494.

\bibitem{AtiyahMacdonald}
M.~F.~Atiyah, I.~G.~Macdonald,
\emph{Introduction to Commutative Algebra},
Addison--Wesley, Reading, MA, 1969.

\bibitem{BuninaKunyavskii2024}
E.~Bunina, B.~Kunyavskii,
\emph{$\Sha$-rigidity of Chevalley groups over local rings},
J. Group Theory \textbf{28} (2025), no.~5, 1143--1161.

\bibitem{BK}
B.~Kunyavski\u\i,
\emph{Local--global invariants of finite and infinite groups: Around Burnside from another side},
Expo. Math. \textbf{31} (2013), 256--273.

\bibitem{BK112}
T.~Ono,
\emph{A note on Shafarevich--Tate sets for finite groups},
Proc. Japan Acad. \textbf{74A} (1998), 77--79.

\bibitem{BK115}
T.~Ono,
\emph{``Shafarevich--Tate sets'' for profinite groups},
Proc. Japan Acad. \textbf{75A} (1999), 96--97.

\bibitem{Sah}
C.-H.~Sah,
\emph{Automorphisms of finite groups},
J. Algebra \textbf{10} (1968), 47--68; addendum: \textit{ibid.} \textbf{44} (1977), 573--575.

\bibitem{Carter}
R.~W.~Carter,
\emph{Simple Groups of Lie Type},
Wiley, London, 1972.

\bibitem{Steinberg}
R.~Steinberg,
\emph{Lectures on Chevalley Groups},
University Lecture Series, Vol.~66, Amer. Math. Soc., Providence, RI, 2016.

\bibitem{v43}
N.~A.~Vavilov,
\emph{Structure of Chevalley groups over commutative rings},
in: \emph{Proc. Conf. Non-associative algebras and related topics (Hiroshima, 1990)},
World Sci. Publ., 1991, pp.~219--335.

\bibitem{VavPlotk1}
N.~A.~Vavilov, E.~B.~Plotkin,
\emph{Chevalley groups over commutative rings. I. Elementary calculations},
Acta Appl. Math. \textbf{45} (1996), 73--115.

\bibitem{M}
H.~Matsumoto,
\emph{Sur les sous-groupes arithm\'etiques des groupes semi-simples d\'eploy\'es},
Ann. Sci. \'Ecole Norm. Sup. (4) \textbf{2} (1969), 1--62.

\bibitem{v41}
G.~Taddei,
\emph{Normalit\'e des groupes \'el\'ementaires dans les groupes de Chevalley sur un anneau},
Contemp. Math. \textbf{55} (1986), Part~II, 693--710.

\bibitem{Vas}
L.~N.~Vaserstein,
\emph{On normal subgroups of Chevalley groups over commutative rings},
T\^ohoku Math. J. (2) \textbf{38} (1986), no.~3, 219--230.

\bibitem{Cn}
P.~Cohn,
\emph{On the structure of the $\GL_2$ of a ring},
Publ. Math. Inst. Hautes \'Etudes Sci. \textbf{30} (1966), 365--413.

\bibitem{Humphreys}
J.~E.~Humphreys,
\emph{Introduction to Lie Algebras and Representation Theory},
Graduate Texts in Mathematics, Vol.~9, Springer-Verlag, New York, 1972.

\bibitem{Abe1993}
E.~Abe,
\emph{Automorphisms of Chevalley groups over commutative rings},
Algebra i Analiz \textbf{5} (1993), no.~2, 74--90;
English translation: St. Petersburg Math. J. \textbf{5} (1994), no.~2, 287--303.

\bibitem{AbeHurley1988}
E.~Abe, J.~Hurley,
\emph{Centers of Chevalley groups over commutative rings},
Comm. Algebra \textbf{16} (1988), no.~1, 57--74.

\bibitem{Humphreys1975}
J.~E.~Humphreys,
\emph{Linear Algebraic Groups},
Springer-Verlag, New York--Heidelberg--Berlin, 1975.

\bibitem{BK21}
W.~Burnside,
\emph{Theory of Groups of Finite Order},
2nd ed., Cambridge Univ. Press, Cambridge, 1911;
reprinted by Dover Publications, New York, 1955.

\bibitem{BK22}
W.~Burnside,
\emph{On the outer automorphisms of a group},
Proc. Lond. Math. Soc. \textbf{11} (1913), 40--42.

\end{thebibliography}
\end{document}